\documentclass[10pt]{amsart}
\usepackage[latin1]{inputenc}
\usepackage{amsmath, amsthm, amssymb}

\usepackage{hyperref}

\usepackage{fancyhdr}
\usepackage{cleveref}
\usepackage{graphicx}

\numberwithin{equation}{section}
\numberwithin{figure}{section}

\allowdisplaybreaks
\theoremstyle{plain} \newtheorem{theorem}{Theorem}[section]
\theoremstyle{plain} \newtheorem{lemma}[theorem]{Lemma}
\theoremstyle{plain} \newtheorem{prop}[theorem]{Proposition}
\theoremstyle{plain}\newtheorem{cor}[theorem]{Corollary}
\theoremstyle{plain} 
\theoremstyle{plain} 
\theoremstyle{definition} \newtheorem{defn}[theorem]{Definition}
\theoremstyle{remark} \newtheorem{rem}[theorem]{Remark}
\theoremstyle{remark} \newtheorem{ex}[theorem]{Example}
\theoremstyle{plain} \newtheorem{question}[theorem]{Question}

\DeclareSymbolFont{bbold}{U}{bbold}{m}{n}
\DeclareSymbolFontAlphabet{\mathbbold}{bbold}
\newcommand{\ind}{\mathbbold{1}}

\DeclareMathOperator{\Prb}{\mathbb{P}}

\def\cJ{\mathcal{J}}
\def\cK{\mathcal{K}}
\def\cP{\mathcal{P}}
\def\cS{\mathcal{S}}
\def\cT{\mathcal{T}}
\def\cW{\mathcal{W}}

\def\DD{\mathbf{D}}
\def\EE{\mathbf{E}}

\def\LL{\mathbf{L}}

\def\TT{\mathbf{T}}
\def\SS{\mathbf{S}}

\def\VV{\mathbf{V}}
\def\WW{\mathbf{W}}

\def\bC{\mathbb{C}}
\def\bE{\mathbb{E}}
\def\bN{\mathbb{N}}
\def\bP{\mathbb{P}}
\def\bC{\mathbb{C}}
\def\bR{\mathbb{R}}
\def\bS{\mathbb{S}}
\def\bT{\mathbb{T}}
\def\bU{\mathbb{U}}
\def\bZ{\mathbb{Z}}

\begin{document}

\title[Recovering a tree from subtree lengths]
{Recovering a tree from the lengths of subtrees spanned by a
randomly chosen sequence of leaves}

\author{Steven N. Evans}
\address{Department of Statistics\\
         University  of California\\ 
         367 Evans Hall \#3860\\
         Berkeley, CA 94720-3860 \\
         U.S.A.}

\email{evans@stat.berkeley.edu}
\thanks{SNE supported in part by NSF grant DMS-0907630
and NIH grant 1R01GM109454-01.}

\author{Daniel Lanoue}
\address{Department of Mathematics\\
         University  of California\\
         970 Evans Hall \#3840\\
         Berkeley, CA 94720-3840\\
         U.S.A.}

\email{dlanoue@math.berkeley.edu}



\date{\today}

\begin{abstract}
Given an edge-weighted tree $\TT$ with $n$ leaves, 
sample the leaves uniformly at random
without replacement and let $W_k$, $2 \le k \le n$, 
be the length of the subtree spanned by the first $k$ leaves.  We consider
the question, ``Can $\TT$ be identified (up to isomorphism) 
by the joint probability distribution of the random vector $(W_2, \ldots, W_n)$?''
We show that if $\TT$ is known {\em a priori} to belong to one of
various families of edge-weighted trees, then the answer is, ``Yes.''  
These families include the edge-weighted trees with edge-weights in general
position, the ultrametric edge-weighted trees, and certain families
with equal weights on all edges such as $(k+1)$-valent and rooted $k$-ary 
trees for $k \ge 2$ and caterpillars.
\end{abstract}

\maketitle

\tableofcontents

\section{Introduction}
\label{sec:Introduction}

\subsection{Background and motivation}
What features of an edge-weighted tree identify it
uniquely up to isomorphism, perhaps within
some class of such trees?  Here
an {\em edge-weighted tree}
is a connected, acyclic finite graph $\TT$ with vertex set 
$\VV(\TT)$ and edge set $\EE(\TT)$ which is equipped with
a function $\WW_\TT: \EE(\TT) \to \bR_{++} := (0,\infty)$.  
The value of $\WW_\TT(e)$
for an edge $e \in \EE(\TT)$ is called the {\em weight} or the
{\em length} of $e$.  Two such trees $\TT'$ and $\TT''$
are isomorphic if there is a bijection 
$\sigma: \VV(\TT') \to \VV(\TT'')$ such that:
\begin{itemize}
\item
$\{u,v\} \in \EE(\TT')$ 
if and only if $\{\sigma(u), \sigma(v)\} \in \EE(\TT'')$,
\item
$\WW_{\TT'}(\{u,v\}) = \WW_{\TT''}(\{\sigma(u), \sigma(v)\})$
for all $\{u,v\} \in \EE(\TT')$.
\end{itemize}
The question above is, more formally, one of asking
for a given class of edge-weighted trees $\bT$ about the possible
sets $\bU$ and functions $\Phi: \bT \to \bU$ such that
for all $\TT',\TT'' \in \bT$ we have
$\Phi(\TT') = \Phi(\TT'')$ if and only if $\TT'$
and $\TT''$ are isomorphic.  If the class $\bT$ consists of
edge-weighted trees for which all edges have length $1$
(we will call such objects {\em combinatorial trees} for the sake
of emphasis), then determining whether two trees in $\bT$
are isomorphic is just a particular case of the
standard graph isomorphism problem.  The general graph
isomorphism problem has been the subject of
a large amount of work in combinatorics and 
computer science -- \cite{MR0485586} already
speaks of the ``graph isomorphism disease'' --
and, in particular, there are many results
on reconstructing the isomorphism type of a graph
from the isomorphism types of subgraphs of various sorts
(see, for example, the review \cite{MR1161466}).
There is also a substantial volume of somewhat parallel research 
on graph isomorphism in computational chemistry (see, for example,
\cite{MR3052391} for a  review).  There seems to be
considerably less work on determining isomorphism
(in the obvious sense) of edge-weighted graphs; of course, in order
for two edge-weighted graphs to be isomorphic the underlying
combinatorial graphs must be isomorphic, but this does not imply
that the best way for checking that two edge-weighted graphs
are isomorphic proceeds by first determining whether the
underlying combinatorial graphs are isomorphic and then somehow
testing whether some isomorphism of the combinatorial graphs
is still an isomorphism when the edge-weights are considered.

We begin with a discussion of previous results that
address various aspects of the problem of determining
when two edge-weighted or combinatorial trees are isomorphic.

A result in \cite{MR0332540} gives the following
 criterion for a bijection
$\sigma: \VV(\TT') \to \VV(\TT'')$, where $\TT'$
and $\TT''$ are combinatorial trees, to be an isomorphism: \\
if $v_0, v_1, \ldots, v_m$ is any sequence from
$\VV(\TT') \sqcup \VV(\TT'')$ such that
$v_0 = v_m$
and
\[
\{v_i, v_j\} \in \EE(\TT') 
\sqcup \EE(\TT'') 
\sqcup \{\{u,\sigma(u)\} : u \in \VV(\TT')\}
\Longleftrightarrow i - j \equiv \pm 1 \mod m,
\]
then $m=4$.

The above result is elegant, but, of course, one does not need to apply
it to all possible bijections to determine whether two
combinatorial trees are isomorphic: there is a
much more explicit and efficient procedure, which we now
describe for the sake of completeness.  First of all,
suppose that $\TT'$ and $\TT''$ have distinguished vertices
$\rho'$ and $\rho''$ and, in addition to the requirements
in the above definition of
an isomorphism $\sigma$, we require that $\sigma$
maps $\rho'$ to $\rho''$; that is, we have rooted trees
and we require that an isomorphism maps the root of one
tree to the root of the other.  The presence of a root
allows us to think of a combinatorial tree as a directed
graph, where the head of an edge is the vertex that is closer
to the root and the tail is the vertex farther from the root.
The children of a vertex are the adjacent vertices that are
farther from the root and, more generally, the descendants
of a vertex $u$ are those vertices $v$ such that the path from
the root to $v$ passes through $u$.  The subtree spanned by
a vertex $u$ and its descendants contains no other vertices
and can be thought of as a combinatorial tree rooted at $u$,
and we call this subtree the subtree below $u$.
Then, two rooted, combinatorial trees $\TT'$ and $\TT''$ 
are isomorphic if the two roots
have the same number of children, say $m$, and there is
an ordering of these children for each tree such that
the subtree below the $i^{\mathrm{th}}$ child of
the root of $\TT'$ is isomorphic (as a rooted, combinatorial
tree) to the subtree below the $i^{\mathrm{th}}$ child of
the root of $\TT''$.  This observation can be turned into
an efficient  algorithm (see, for example, \cite[Example 3.2]{MR0413592}).
Now, two combinatorial trees are isomorphic if there
is some choice of roots such the resulting rooted, combinatorial
trees are isomorphic.  A {\em center} of a combinatorial tree
is a vertex $c$ such that 
\[
\max_{v \in \VV(\TT)} r_\TT(c,v) 
=
\min_{u \in \VV(\TT)} \max_{v \in \VV(\TT)} r_\TT(u,v),
\]
where $r_\TT(u,v)$ is the number of edges in he unique
path between $u$ and $v$ for $u,v \in \VV(\TT)$, and
a combinatorial tree has either a unique center or two centers that
are adjacent.  It is therefore possible to determine if
two combinatorial trees are isomorphic by rooting each of them at
their various centers and checking if any two such rooted, 
combinatorial trees are isomorphic.

We, however, are interested in whether there are ``statistics'' of
a more numerical character that can be used to decide tree isomorphism.
For combinatorial trees, one somewhat obvious possibility is the 
multiset of eigenvalues of some matrix associated with the tree such
as the adjacency matrix or the distance matrix.  Unfortunately, the
results of \cite{schwenk, MR1231010,
MR750401, MR1609509, matsen2011ubiquity, MR2956206} show that
not only is the isomorphism type of a tree not uniquely determined by
the spectrum of its adjacency matrix but for various
ensembles of combinatorial trees if one picks a tree uniformly at random
from those in the ensemble with $n$ vertices, then the probability
there is another tree in the ensemble with an adjacency matrix
that has the same spectrum converges to one as $n \to \infty$.
The results of \cite{matsen2011ubiquity} can be used to
show that an analogous phenomenon is present when one considers the
spectrum of the matrix of leaf-to-leaf distances.

Two trees have adjacency matrices with
the same spectrum if and only if the characteristic polynomials of
the adjacency matrices are equal.  Given some irreducible representation
of the symmetric group on the number of letters equal to the dimension
of a square matrix, the immanantal polynomial of the matrix is
constructed in the same manner as the characteristic polynomial
except that the determinant is replaced by a similarly defined
object for which the sign character is replaced by the character
of the representation.  One might hope that the immanantal polynomials 
are more successful at deciding isomorphism of combinatorial trees,
but a result of \cite{MR1231010} shows that if the adjacency matrices
of two combinatorial trees have the same characteristic polynomials,
then they have the same immanantal polynomials for every irreducible
representation.  We note that \cite{MR0228370} already contains
an example of two combinatorial trees with adjacency matrices
that are explicitly shown to have the same immanantal polynomial.

The greedoid Tutte polynomial of a combinatorial tree $\TT$
encodes for each $i$ and $\ell$ the number of subtrees
of $\TT$ that have $i$ internal vertices and $\ell$ leaves.
It was conjectured in \cite{MR1369283} that this information
identifies the isomorphism type of a combinatorial tree.
However, it was shown in \cite{MR2234989} that there are infinitely many
pairs of nonisomorphic caterpillars that share the
same greedoid Tutte polynomial: a caterpillar is 
a combinatorial tree that consists of some number
of internal vertices along a single path and leaves
that are each adjacent to one of the internal vertices.
This contrasts with the situation for rooted, combinatorial
trees; it is shown in \cite{MR967486} that there is
a two-variable polynomial defined for all rooted, directed
graphs (and hence, in particular, for rooted, combinatorial
trees) such that two rooted, combinatorial trees
have the same polynomial if and only if they are isomorphic.
The polynomial in \cite{MR967486} is defined recursively,
but it is not hard to see that it encodes in a compact manner the total
number of vertices in the tree, the number of children
of the root, the number of vertices in each of the subtrees
below the children of the root, and so on.

The chromatic symmetric function of a graph was introduced
in \cite{MR1317387}.  A proper coloring of a finite graph
is a function $\kappa$ from the vertices of the graph to $\bN$ such
that adjacent vertices are assigned different values.
We can introduce an equivalence relation on the proper colorings
by declaring that two colorings $\kappa'$ and $\kappa''$
are equivalent if there is a bijection $\pi : \bN \to \bN$
such that $\kappa'' = \pi \circ \kappa'$.  For a graph
with $m$ vertices, each equivalence
class gives rise to a partition 
$\lambda_1 \ge \lambda_2 \ge \ldots \ge \lambda_k > 0$ of $m$ by taking, 
for any $\kappa$ in the equivalence class,
$\lambda_i$ to be the $i^{\mathrm{th}}$ largest of the cardinalities
$\#\{v : \kappa(v) = j\}$ as $j$ ranges over $\bN$.  The
chromatic symmetric function encodes for each partition of $m$
the number of equivalence classes of colorings that give rise
to that partition.  It was conjectured in \cite{MR1317387}
that nonisomorphic combinatorial 
trees have distinct chromatic symmetric functions.
It was shown in \cite{MR2382514} that this conjecture is true
for caterpillars and that paper also reports on computational
results verifying that the conjecture holds for the class
of trees with at most $23$ vertices.  Further work
related to the conjecture for the special
case of trees with a single centroid is contained in \cite{MR3147202}.

Our point of departure in this paper is the well-known fact 
\cite{Zaretskii_65, MR0237362, buneman1971recovery, MR0363963} that
an edge-weighted tree can be reconstructed from its matrix
of leaf-to-leaf distances (see \cite{felsenstein} for
an indication of the importance of this observation
in the statistical reconstruction of phylogenetic trees).
In fact, an edge-weighted tree with $n$ leaves can be reconstructed from the
collection of total lengths of subtrees spanned by all
subsets of $m$ leaves provided $n \ge 2m-1$ \cite{MR2064171}.
We remark that the total length of the subtree spanned by a set of leaves
is an important quantity in phylogenetics where it is called the phylogenetic
diversity of the corresponding set of taxa \cite{MR2359353}.

Given these results, one might imagine that the multiset of 
leaf-to-leaf distances suffices to identify 
the isomorphism type of an edge-weighted tree.  This is
certainly not the case.  For example, consider the
two combinatorial caterpillars $\TT'$ and $\TT''$ with $28$ leaves each, where
$\TT'$ has $3$ internal vertices $a',b',c'$ in order along a path that are adjacent respectively to $2,11,12$ leaves, and
$\TT''$ has $3$ internal vertices  $a'',b'',c''$ in order along a path that are
adjacent respectively to $3,14,8$ leaves.  Taking the $\binom{25}{2}$ pairs of distinct leaves in $\TT'$, 
we see that the distance $2$ appears 
$\binom{2}{2} + \binom{11}{2} + \binom{12}{2} = 122$ times,
the distance $3$ appears $2 \times 11 + 11 \times 12 = 154$ times,
and the distance $4$ appears $2 \times 12 = 24$ times.  Similarly,
taking the $\binom{25}{2}$ pairs of distinct leaves in $\TT''$, 
we see that the distance $2$ appears 
$\binom{3}{2} + \binom{14}{2} + \binom{8}{2} = 3 + 91 + 28 = 122$ times,
the distance $3$ appears $3 \times 14 + 14 \times 8 = 154$ times,
and the distance $4$ appears $3 \times 18 = 24$ times.
Probabilistically, we have just shown that if we pick two leaves
uniformly at random without replacement
from an edge-weighted tree, then the isomorphism
type of the tree is not uniquely identified by the probability
distribution of the distance between the two leaves.

Note in this last example that if we looked at the multisets of
lengths of subtrees spanned by three leaves, then we would see the length $3$
appearing $\binom{11}{3}+\binom{12}{3} = 335$ times for $\TT'$
and $\binom{3}{3} + \binom{14}{3}+\binom{8}{3} = 421$ times for $\TT''$,
and hence the probability distribution of the length of the subtree
spanned by three leaves chosen uniformly at random is not the same
for the two trees.

In order to proceed further, we need to introduce some
more notation.
Write $\LL(\TT)$ for the set of leaves
of an edge-weighted tree $\TT$.  Given
a subset $K$ of $\LL(\TT)$, let $\WW_\TT(K)$ be the length of the
subtree spanned by $K$; that is, $\WW_\TT(K)$ is the sum of the lengths
of the edges in the smallest connected subgraph of $\TT$ with a vertex set that
contains $K$.

It is possible to calculate the total length of $\TT$, 
that is, $\WW_\TT(\LL(\TT))$, using the following result from
\cite{MR2053839} that extends one for the
special case of $3$-valent trees in \cite{pauplin2000direct}.
Write $d_\TT(v)$ for the degree of an interior vertex $v$ of $\TT$
(that is, $v \in \VV(\TT) \setminus \LL(\TT)$).  For distinct leaves
$x,y \in \LL(\TT)$ denote by $I_\TT(x,y)$ the set of interior
vertices on the unique path in $\TT$ between $x$ and $y$ and
put 
\[
h_\TT(x,y) := \prod_{v \in I_\TT(x,y)} ((d_\TT(v) - 1)!)^{-1}.
\]
Let $r_\TT(x,y)$ be the sum of the lengths of the edges in
the path between $x$ and $y$.  Then,
\[
\WW_\TT(\LL(\TT)) 
= 
\sum_{\{x,y\} \subseteq \LL(\TT), x \ne y} h_\TT(x,y) r_\TT(x,y).
\]
Of course, a similar formula gives  $\WW_\TT(K)$ for any 
$K \subseteq \LL(\TT)$; the path
between a pair of leaves of the subtree is the same
as the path between them in $\TT$, the length of this
path is the same in the subtree as it is in $\TT$,
but the degree of an interior
vertex of the subtree can be less
than its degree as an interior vertex of $\TT$.

Suppose that $\# \LL(\TT) = n$
and $Y_1, \ldots, Y_n$ is a uniformly
distributed random listing of $\LL(\TT)$;
that is, $Y_1, \ldots, Y_n$ is the result
of sampling the leaves of $\TT$ uniformly
at random without replacement.  Set 
$W_k := \WW_\TT(\{Y_1, \ldots, Y_k\})$
for $2 \le k \le n$; that is, the random
variable $W_k$ is the length of the
subtree spanned by the first $k$ of the
randomly chosen leaves.  We write
$\cW_\TT$ for the $(n-1)$-dimensional random vector
$(W_2, \ldots, W_n)$ and call this random vector 
the {\em random length sequence} of $\TT$. 

In this paper we address the following question.

\begin{question}{Can we reconstruct the edge-weighted tree
$\TT$ up to isomorphism from the
joint probability distribution of the random length sequence $\cW_{\TT}$?}
\label{q:main}
\end{question}

Another way of framing this question is the following.
Write $y_1, \ldots, y_n$ for the leaves of $\TT$
and let $\cJ_\TT$ be the multiset with cardinality $n!$ that results
from listing the $(n-1)$-dimensional vectors 
\[
(\WW_\TT(\{y_{\pi(1)},y_{\pi(2)}\}), 
\WW_\TT(\{y_{\pi(1)},y_{\pi(2)},y_{\pi(3)}\}), 
\ldots, \WW_\TT(\{y_{\pi(1)}, \ldots ,y_{\pi(n)}\}))
\]
as $\pi$ ranges of the permutations of $[n]:=\{1,\ldots,n\}$.
We stress that $\cJ_\TT$ is a multiset; that is, we do not know which
increasing sequences of lengths go with which ordered listings of
the leaves.

\begin{question}{ Can we reconstruct the edge-weighted tree $\TT$ 
up to isomorphism from the multiset of length sequences $\cJ_\TT$?}
\end{question}

We end this section with some remarks about the problem of reconstructing trees
from various so-called {\em decks}, as this subject has some similarities
to the questions we consider.  In \cite{MR0120127}, Ulam asked whether it
is possible to reconstruct the isomorphism type of a graph 
with at least $3$ vertices from the isomorphism types of the subgraphs
obtained by deleting each of the vertices.  This question was resolved
in the affirmative for combinatorial trees in \cite{MR0087949}.  Moreover,
later results established that it is not necessary to know the forests obtained
by deleting every vertex.  For example, it was shown in \cite{MR0200190}
that it suffices to know the subtrees obtained by deleting leaves.
This latter result was strengthened in \cite{MR0256926}, where it
was found that it is only necessary to know which nonisomorphic forests
are obtained and not what the multiplicity of each isomorphism type is,
and in \cite{MR0260614}, where it was shown
that it suffices to take only those leaves $p$ that are {\em peripheral}
in the sense that
\[
\max_{v \in \VV(\TT)} r_\TT(p,v) 
= 
\max_{v \in \VV(\TT)} \max_{v \in \VV(\TT)} r_\TT(u,v).
\]
Along the same lines, it was established in \cite{MR680306} 
that it is enough to take only 
the nonleaf vertices, provided that there are at least three of them.
The line of inquiry in \cite{MR786484} is the most similar to ours:
an example was presented of two trees for which the respective sets of
vertices may be paired up
in such a way that for each pair
the sizes of the trees in the forests produced by removing
each element of the pair from its tree are the same, 
and a necessary and sufficient condition was given
for a tree to be uniquely reconstructible from this sort of data, which the
authors of \cite{MR786484} call the {\em number deck} of the tree.


\subsection{Overview of the main results}

We will answer Question~\ref{q:main} in the affirmative for 
a few different classes of trees.  
Some classes will have general edge-weights and some classes
will be combinatorial trees.
It is clear that in the case of general edge-weights we must
restrict to trees that have no vertices
with degree $2$ because otherwise we can subdivide
any edge into arbitrarily many edges with the same total
length and the joint probability distribution of the random length sequence
will be unchanged -- see \Cref{fig:NonSimple cex}.   We call such trees 
{\em simple}.  The terms irreducible or homomorphically irreducible are also
used in the literature. 


\begin{figure}[ht]
    \centering
    \includegraphics[width=0.4\textwidth]{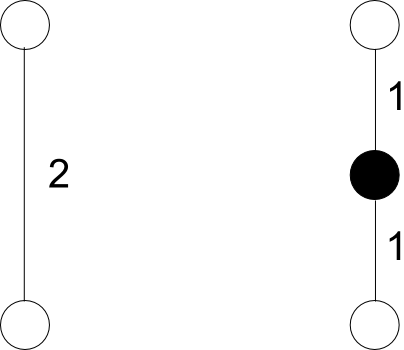}
    \caption{Two non-isomorphic edge-weighted trees that 
    cannot be distinguished by the joint probability distribution
    of their random length sequences.}
    \label{fig:NonSimple cex}
\end{figure}

 
Our first result is for the the class of {\em stars}; that is,
edge-weighted trees with $n \ge 3$ leaves that have
a single interior vertex.  Note that such trees are simple.
For any edge-weighted tree with $n$ leaves, $W_n$ is a constant
(the total length of the tree) and $W_n - W_{n-1}$ is a uniformly
distributed random pick from the lengths of the $n$ edges that are
adjacent to one of the leaves.  The following simple result is immediate
from this observation.

\begin{theorem}{For $n \ge 3$ the isomorphism type of
a star is uniquely determined 
by the joint probability distribution of its random length sequence.}
\label{thrm:main star}
\end{theorem}

The simple trees with two leaves all consist of a single edge and have
a random length sequence $(W_2)$, where $W_2$ is the length of that edge,
and so the isomorphism type of such a tree is uniquely determined
by the joint probability distribution of
its random length sequence.  The simple trees with three leaves
are stars, and it follows from \Cref{thrm:main star} that the isomorphism
type of such a tree is uniquely determined
by the the joint probability distribution of its random length sequence.

We next consider simple, edge-weighted trees with four leaves.

\begin{theorem}{For $2 \le n \le 4$, the isomorphism type
of a simple, edge-weighted tree 
$\TT$ with $n$ leaves is uniquely determined 
by the joint probability distribution of its random length sequence.}
\label{thrm:main n4}
\end{theorem} 

The proof of this result is via consideration of possible cases.  
Similar proofs could be attempted for larger numbers of leaves,
but the main reason we include the result is to show how such a proof
for even a small number of leaves leads to quite a few cases and
because we will need the case of four leaves later.

It is well-known that any simple, combinatorial tree with labeled
leaves can be reconstructed
from the simple, combinatorial trees spanned by each subset of 
four leaves (the so-called quartets) \cite[Theorem~6.3.7]{MR2060009}.
With this and \Cref{thrm:main n4} in mind, one might imagine that the
isomorphism type of simple, edge-weighted tree can be determined
from the joint probability distribution of $(W_2, W_3, W_4)$. 
However, putting such a strategy into practice
would seem to be rather complicated because there can be two sets
of leaves $\{y_1', y_2', y_3', y_4'\}$
and
$\{y_1'', y_2'', y_3'', y_4''\}$
such that
$\{y_1', y_2', y_3', y_4'\} \ne  \{y_1'', y_2'', y_3'', y_4''\}$ 
but
$\WW_\TT(\{y_1', y_2'\}) = \WW_\TT(\{y_1'', y_2''\})$, 
$\WW_\TT(\{y_1', y_2', y_3'\}) = \WW_\TT(\{y_1'', y_2'', y_3''\})$, 
and
$\WW_\TT(\{y_1', y_2', y_3',y_4'\}) = \WW_\TT(\{y_1'', y_2'', y_3'',y_4''\})$.
One way of ruling out such annoying algebraic coincidences is to assume that
the edge-weighted tree $\TT$ has {\em edge-weights in general position}, by which we
mean that the sums of the lengths of any two different 
(not necessarily disjoint) subsets of edges of $\TT$ are not equal. 

\begin{theorem}{The isomorphism type
of a simple, edge-weighted tree $\TT$ with edge-weights 
in general position is uniquely determined by the 
joint probability distribution of its random length sequence.}
\label{thrm:main generalposition}
\end{theorem}

The last family of edge-weighted trees with general edge-weights
whose elements we can identify up to isomorphism
from the joint probability distributions of their random length sequences
is the class of {\em ultrametric} trees. For the
sake of completeness, we now define this class.
Recall that for leaves $i, j \in \LL(\TT)$ we denote by $r_\TT(i, j)$ the
distance between them; that is, $r_\TT(i,j)$ is the sum of the lengths
of the edges on the unique path between $i$ and $j$. 
The edge-weighted tree $\TT$ is ultrametric if for any leaves $i, j, k \in \LL(\TT)$ we have
$$
r_\TT(i, k) \leq r_\TT(i, j) \vee r_\TT(j, k),
$$
from which it follows that for any leaves $i, j, k \in \LL(\TT)$
at least two of the distances 
$r_\TT(i, j)$, $r_\TT(i, k)$, and $r_\TT(j, k)$ are equal
while the third is no greater than that common value.
Equivalently, an edge-weighted tree $\TT$
is ultrametric if, when it is thought of as a real tree 
(that is, a metric space where the edges are treated as real intervals of varying lengths given by their edge-weights 
-- see, for example, 
\cite{MR2351587}), then there is a (unique) point $\rho$ 
 called the root (which may be in the interior of an edge)
 such that the distance from $\rho$ to a leaf is the same for all leaves.
We will make use of both definitions.  It is immediate from 
the former definition that the subtree of an ultrametric tree 
spanned of a subset of leaves is itself ultrametric. 

\begin{theorem}{The isomorphism type of 
an ultrametric, simple, edge-weighted tree $\TT$ 
is uniquely determined by the joint probability distribution of
its random length sequence.}
\label{thrm:main ultrametric}
\end{theorem} 

\begin{rem}
{The proof of \Cref{thrm:main ultrametric} establishes an even stronger result.  Namely, the isomorphism type
of an ultrametric, simple, edge-weighted tree $\TT$ 
is uniquely determined by the minimal element of $\cJ_\TT$ in the lexicographic order.}
\label{rem:ultrametric min enough}
\end{rem}

\begin{rem}
{We call attention to a subtle point in the statements of
\Cref{thrm:main generalposition} and
 \Cref{thrm:main ultrametric}.  Both results say that if we
are given the joint probability distribution of the random length
 sequence of an edge-weighted tree $\TT$ -- information
 that certainly includes the number of leaves of $\TT$ --
 and we know, a priori, that
 $\TT$ has a certain extra property (edge-weights in general position
 or ultrametricity), then we can determine the isomorphism type of $\TT$.
 The theorems do not, however, say whether it is possible to determine from
 the joint probability distribution of its random length
 sequence whether a simple, edge-weighted tree $\TT$ has its edge-weights
 in general position or is ultrametric.  
 We do not have results that settle this question, but we say
some more about it in \Cref{rem:more ultrametric determination} and believe it
 is an interesting area for future research.}
\label{rem:ultrametric determination}
\end{rem}

Observe that if $\TT$ is an edge-weighted tree, $a$ is any vertex of $\TT$, and $c$ is
a constant such that $c \ge \max\{r_\TT(a,i) : i \in \LL(\TT)\}$, then
$\tilde r_\TT: \LL(\TT) \times \LL(\TT) \to \bR_+$ defined by
\[
\tilde r_\TT(i,j) := c + \frac{1}{2}(r_\TT(i,j) - r_\TT(a,i) - r_\TT(a,j)), \quad i \ne j,
\]
and
\[
\tilde r_\TT(i,i) := 0,
\]
is an ultrametric on $\LL(\TT)$ that arises from suitable edge-weights on $\TT$.  The metric
$\tilde r_\TT$ is often called the {\em Farris transform} of $r_\TT$ -- see 
\cite{MR2311928} for a review of the many appearances of this object 
in various areas from phylogenetics to metric geometry.
It might be hoped that an affirmative answer to \Cref{q:main} for general
edge-weighted trees will follow from \Cref{thrm:main ultrametric}.  However,
we have been unable to find an argument which shows that the joint probability distribution of
the random length sequence of the tree $\TT$ equipped with the new edge-weights is determined by the joint probability distribution of the random length sequence for the original edge-weights.

Suppose that $\TT$ is a rooted, simple combinatorial tree with root $\rho$. We can define a partial order on $\VV(\TT)$
by declaring that that $x \le y$ if $x$ is on the unique path from $\rho$ to $y$.  Two vertices $x,y \in \VV(\TT)$ have a unique
greatest lower bound in this partial order that we write as $x \wedge y$ and call the {\em most recent common ancestor} of $x$ and $y$.
The map $\hat r_\TT: \LL(\TT) \times \LL(\TT) \to \bR_+$ defined by
\[
\hat r_\TT(i,j) := \#\{k \in \LL(\TT) : i \wedge j < k\}
\]
is an ultrametric on $\LL(\TT)$ and hence it arises from a collection of edge-weights $\hat \WW_\TT$ on $\TT$.
A directed edge $(x,y)$ in $\TT$ with $x \le y$ is necessarily of the form $x = i \wedge j = i \wedge k$ and
$y = j \wedge k$ for some $i,j,k \in \LL(\TT)$.  If $e = \{x,y\}$ is the corresponding undirected edge, then
\[
\begin{split}
\hat \WW_\TT(e) 
& = \frac{1}{2}(\hat r_\TT(i,j) - \hat r_\TT(j,k)) \\
& = \frac{1}{2}(\#\{\ell \in \LL(\TT) : x < \ell\} - \#\{\ell \in \LL(\TT) : y < \ell\}). \\
\end{split}
\]
Therefore, if $\TT'$ is a subtree of $\TT$ spanned by some set of leaves $K \subseteq \LL(\TT)$
and $\DD(\TT')$ is the set of directed edges of $\TT'$, then we have that the length of $\TT'$ is
\[
\begin{split}
\hat \WW_\TT(K) 
& = \frac{1}{2}\sum_{(x,y) \in \DD(\TT')}\left(\sum_{\ell \in \LL(\TT)} \ind\{x < \ell\} - \ind\{y < \ell\}\right) \\
& = \frac{1}{2} \#\{((x,y),\ell) \in \DD(\TT') \times \LL(\TT) : x < \ell, \, y \not < \ell\}. \\
\end{split}
\]
The following result is immediate from \Cref{thrm:main ultrametric} and \Cref{rem:ultrametric min enough}.

\begin{cor}
{The isomorphism type of 
a simple, combinatorial tree $\TT$ 
is uniquely determined by the minimal element of the set $\cJ_\TT$ of length sequences
obtained after designating a root for $\TT$ and equipping
$\TT$ with the edge-weights $\hat \WW_\TT$.}
\end{cor}

We now turn our focus to combinatorial trees 
and drop the assumption of simplicity. That is,
all edge-weights are equal to one and there
may be vertices with degree two.
We answer Question~\ref{q:main} in the affirmative
for two families of combinatorial trees.

\begin{figure}[ht]
    \centering
    \includegraphics[width=0.4\textwidth]{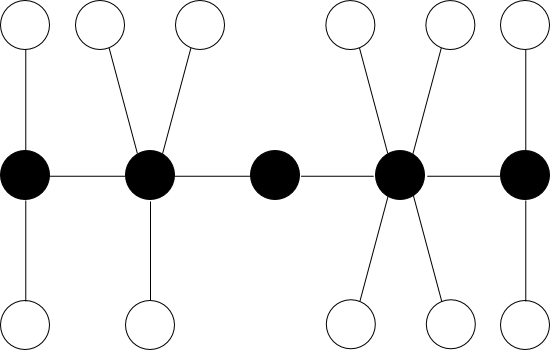}
    \caption{A caterpillar tree.  Removing the leaves (white vertices) results in a path of length $5$ (black vertices).}
    \label{fig:Cat ex}
\end{figure}

First, a combinatorial tree $\TT$ is a {\em caterpillar} if the deletion of the 
leaves along with the edges adjacent to them results in a path with $\ell+1$
vertices (and hence $\ell$ edges) -- see, for example,
\Cref{fig:Cat ex}.  Choose some direction for the
path and number from $0$ to $\ell$
the vertices on the path encountered successively 
in that direction  and write $n_i$ for the number
of leaves adjacent to the vertex numbered $i$. Note that
$n_0 \ge 1$ and $n_\ell \ge 1$.  Two sequences $n_0', \ldots, n_{\ell'}'$
and $n_0'', \ldots, n_{\ell''}''$ correspond to isomorphic trees if and only
if $\ell' = \ell'' = \ell$, say, and either $n_i' = n_i''$, $0 \le i \le \ell$
or $n_i' = n_{\ell-i}''$, $0 \le i \le \ell$.

\begin{theorem}{The isomorphism type of a caterpillar is uniquely 
determined by the joint probability distribution of its random length
 sequence. Furthermore, it is possible to determine from the joint probability
distribution of the random length sequence of a combinatorial tree whether
the tree is a caterpillar.}
\label{thrm:main caterpillar}
\end{theorem}

Our final results are for the classes of (unrooted) 
{\em $(k+1)$-valent} and {\em rooted  $k$-ary} combinatorial trees. For $k \geq 2$, a 
$(k+1)$-valent combinatorial tree is a combinatorial tree for which 
all vertices have degree either $k+1$ (the internal vertices) or $1$
(the leaves). For $k \geq 2$,  a rooted $k$-ary combinatorial tree 
is a combinatorial tree for which one internal vertex (the root) 
has degree $k$ and the remaining internal vertices have degree $k+1$;
the leaves, of course, have degree $1$. 
When $k = 2$ we refer to a rooted $2$-ary combinatorial tree 
as a rooted {\em binary} combinatorial tree. Attaching an extra vertex
via and edge to the root of a rooted $k$-ary tree produces
a $(k+1)$-valent combinatorial tree.

\begin{theorem}{The isomorphism type of a 
$(k+1)$-valent combinatorial tree (respectively, a $k$-ary tree)
 is uniquely determined by the joint probability distribution of its random length sequence.}
\label{thrm:main kary1}
\end{theorem}

In fact, our proof of \Cref{thrm:main kary1} leads us to a stronger conclusion.

\begin{theorem}{Fix $n > 1$. Let $\cT$ be a random 
$(k+1)$-valent combinatorial tree
(respectively, a random $k$-ary combinatorial tree) with $n$ leaves. 
Then, the probability distribution of the isomorphism
type of $\cT$ is uniquely determined by 
the joint probability distribution of its random length sequence. }
\label{thrm:main kary2}
\end{theorem}

Note that in \Cref{thrm:main kary2} there are two sources of randomness
in the construction of the random length sequence: 
we first choose a realization of the random $\cT$ and then take an independent
uniform random listing of the leaves to build the increasing sequence of
subtrees and their lengths.  

 
The rest of the paper consists primarily of proofs of 
the above results in the order we have presented them.
In \Cref{sec:open probs} we briefly discuss 
further open questions related to Question~\ref{q:main}.

\section{Trees with up to $n = 4$ leaves: Proof of \Cref{thrm:main n4}}
\label{sec:n4 trees}

We begin by looking at Question~\ref{q:main} for
edge-weighted trees with a small number of leaves and give a 
proof of \Cref{thrm:main n4} that answers  Question~\ref{q:main}
in the affirmative for general, simple edge-weighted trees with  
$n = 2, 3$ or $4$ leaves.

The case of \Cref{thrm:main n4} for simple trees with $n = 2$ 
leaves is trivial, as  all such trees have two leaves and one edge, 
$\cW_\TT = (W_2)$ in this case, and $W_2$ is the length
of the edge.

The case of $n = 3$ leaves is only slightly more complicated, as all such
trees are star-shaped. Thus, determining $\TT$ from $\cW_{\TT}$ consists of determining its three edge weights. These can be inferred easily from $\cW_{\TT}$ by looking at the distribution of $W_3 - W_2$, which, since $W_3$ is constant (equal to the total length of $\TT$), is distributed as a uniform random choice from the three edge weights.

Finally, we give a proof of \Cref{thrm:main n4} in the case when $n = 4$.

\begin{proof}
For $n = 4$ leaves, there are two possible simple combinatorial trees,
and hence two possibilities for the shape of $\TT$. 
The first is the star-shaped tree with four edges and one interior vertex. The second is the $3$-valent tree with two interior vertices and one interior edge.  See \Cref{fig:n4 trees}.

\begin{figure}[ht]
    \centering
    \includegraphics[width=0.4\textwidth]{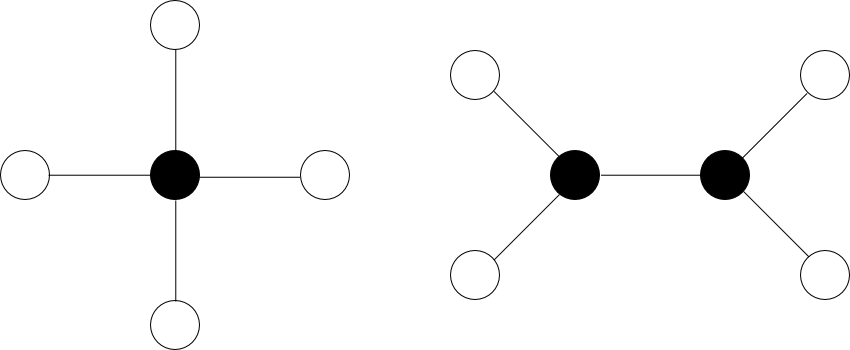}
    \caption{The two possible simple combinatorial trees with $n = 4$ leaves.}
    \label{fig:n4 trees}
\end{figure}

To determine which possibility $\TT$ is, we first look at the distribution of $W_4 - W_3$ to find the lengths of the four edges connecting directly to the four leaves. Call these edges {\em pendent}. If the sum of the four pendent edge lengths equals $W_4$, then $\TT$ is star shaped and we have determined $\TT$ up to isomorphism. If not, then $\TT$ is $3$-valent and the difference between $W_4$ and the sum of the pendent edge lengths is the length $e$ of the interior edge. All that is left to determine $\TT$ up to isomorphism in this second case is determining how the pendent edges pair on each side of the interior edge.

First, if the multiset of the lengths of pendent edges is of the form 
$\{a, a, a, a\}$ or $\{a, a, a, b \}$, 
then $\TT$ is already uniquely determined.

Next, if the multiset is of the form $\{a, a, b, b \}$, 
then we need to distinguish between the case where the 
leaves with pendent edges of length $a$ are siblings 
(and thus so are the leaves with pendent edge length $b$) 
and the case where leaves with pendent edge lengths $a$ and $b$ are paired. 
In the former case the possible values of $W_2$ are
$a+a, b+b, a+b+e$
with respective probabilities
$\frac{4}{24}, \frac{4}{24}, \frac{16}{24}$,
 whereas in the latter case the 
 possible values of $W_2$ are
$a+b, a+b+e, a+a+e, b+b+e$ 
with respective probabilities
$\frac{8}{24}, \frac{8}{24}, \frac{4}{24}, \frac{4}{24}$,
and we can certainly distinguish between the two cases.

If the multiset of pendent edge lengths is of the form $\{a, a, b, c\}$, 
then there are the following two possibilities:
\begin{itemize}
\item[(P1)]
the two leaves with pendent edge length 
$a$ are siblings and the two leaves with pendant edge lengths $b$ and $c$
are siblings, 
in which case the possible values of $W_2$ are
$a+a, a+b+e, a+c+c, b+c$ with respective probabilities 
$\frac{4}{24}, \frac{8}{24}, \frac{8}{24}, \frac{4}{24}$;
\item[(P2)]
a leaf with pendent edge length $a$ is the sibling of the one 
with pendent edge length $b$ and the other leaf with pendent edge length $a$
is the sibling of the one with pendent edge length $c$,
in which case the possible values of $W_2$ are
$a+b, a+e, a+a+e, a+b+e, a+c+c, b+c+e$ with respective probabilities 
$\frac{4}{24}, \frac{4}{24}, \frac{4}{24}, \frac{4}{24}, \frac{4}{24}, \frac{4}{24}$. 
\end{itemize} 

Suppose without loss of generality that $b<c$.  
If $a<b<c$, then $\bP\{W_2 = a+a\}$ is 
$\frac{4}{24}$ for (P1) and $0$ for (P2). 
If $b<a<c$ or $b<c<a$, then $\bP\{W_2 = a+c+e\}$ is 
$\frac{8}{24}$ for (P1) and $\frac{4}{24}$ for (P2).
In all cases we can distinguish between (P1) and (P2).  

Finally, if the multiset of pendent edge lengths is of the form
$\{a, b, c, d \}$,
then there are the following two possibilities:
\begin{itemize}
\item[(P3)]
the leaf with pendent edge length $a$ 
is paired with the one with pendent edge length $b$ 
and the leaf with pendent edge length $c$ 
is paired with the one with pendent edge length $d$,
in which case the possible values of $W_2$ are
$a+b, c+d, a+c+e, a+d+e, b+c+e, b+d+e$
with common probability $\frac{4}{24}$; 
\item[(P4)]
the leaf with pendent edge length $a$ 
is paired with the one with pendent edge length $c$ 
and the leaf with pendent edge length $b$ 
is paired with the one with pendent edge length $d$,
in which case the possible values of $W_2$ are
$a+c, b+d, a+b+e, a+d+e, b+c+e, c+d+e$
with common probability $\frac{4}{24}$;
\item[(P5)]
the leaf with pendent edge length $a$ 
is paired with the one with pendent edge length $d$ 
and the leaf with pendent edge length $b$ 
is paired with the one with pendent edge length $c$,
in which case the possible values of $W_2$ are
$a+d, b+c, a+c+e, a+b+e, c+d+e, b+d+e$
with common probability $\frac{4}{24}$;
\end{itemize}

Suppose without loss of generality that $a < b < c < d$. Then
possibility (P3) holds if and only if $\bP\{W_2=a+b\} > 0$ and
possibility (P5) holds if and only if $\bP\{W_2=a+b\} = 0$ and
$\bP\{W_2 = b+d+e\}>0$, so we can distinguish between (P3), (P4)
and (P5).
\end{proof}

The  argument in the proof of
\Cref{thrm:main n4} seems rather {\em ad hoc} and it does not
suggest a systematic approach to obtaining the analogous result
for trees with an arbitrary numbers of leaves.  
The number of simple combinatorial trees
with $n$ leaves grows so rapidly with $n$ (see, for example, 
\cite{felsenstein}) that even for trees with a relatively small
fixed number of leaves a case-by-case argument seems rather forbidding.  
Nonetheless, we do conjecture that an affirmative answer to Question~\ref{q:main} 
holds more generally.

\section{Trees in general position: Proof of \Cref{thrm:main generalposition}}
\label{sec:general pos}

Recall that the edge-weights of a simple, edge-weighted tree $\TT$ are
in general position if the sum of the lengths of any two distinct 
subset of edges of $\TT$ are not equal. 

\begin{proof}
By assumption, if 
$\{y_1', \ldots, y_k'\}$ and $\{y_1'', \ldots, y_k''\}$
are two subsets of $\LL(\TT)$ such that
$\WW_\TT(\{y_1', \ldots, y_k'\}) = \WW_\TT(\{y_1'', \ldots, y_k''\})$,
then $\{y_1', \ldots, y_k'\} = \{y_1'', \ldots, y_k''\}$.  
Consequently, if 
$\{y_1', \ldots, y_k'\}$ and $\{y_1'', \ldots, y_k''\}$
are two subsets of $\LL(\TT)$ such that
$\WW_\TT(\{y_1', \ldots, y_j'\}) = \WW_\TT(\{y_1'', \ldots, y_j''\})$
for $2 \le j \le k$, then 
$\{y_1', y_2'\} = \{y_1'', y_2''\}$
and 
$y_j' = y_j''$ for $3 \le j \le k$.

Recall that $Y_1, \ldots, Y_n$ are the successive randomly chosen leaves
used in the construction of $\cW_\TT = (W_2, \ldots, W_n)$.

Because $W_n - W_{n-1}$ is the length of the pendent edge
attaching $Y_n$ to the rest of $\TT$, it follows that the set
$C := \{\ell > 0 : \bP\{W_n - W_{n-1} = \ell\} > 0\}$
has $n$ elements and  $\bP\{W_n - W_{n-1} = \ell\}=\frac{1}{n}$
for each $\ell \in P$.  There are at least two leaves of $\TT$
that are siblings, and so there exist $\ell', \ell'' \in C$
such that $\bP\{W_2 = \ell' + \ell''\} > 0$.  Fix such a pair
of lengths and write
$x_1$ and $x_2$ for the (unique) leaves of $\TT$
with pendent edges having respective lengths $\ell'$ and $\ell''$.
We have  $\bP\{W_2 = \ell' + \ell''\} = \frac{1}{\binom{n}{2}}$,
and the event $\{W_2 = \ell' + \ell''\}$ coincides with
the event $\{\{Y_1,Y_2\} = \{x_1, x_2\}\}$.

By assumption, the set
$D := 
\{\ell > 0 : \bP\{W_3 - W_2 = \ell \, | \, W_2 = \ell' + \ell''\} > 0\}$
has $n-2$ elements and  
$\bP\{W_3 - W_2 = \ell \, | \, W_2 = \ell' + \ell''\} = \frac{1}{n-2}$
for each $\ell \in D$.  Index the values of $D$ as $\ell_3, \ldots, \ell_n$
and write $x_k$, $3 \le k \le n$, for the unique leaf of $\TT$
that is distance $\ell_k$ from the unique vertex of $\TT$
that is adjacent to both
of the sibling leaves $x_1$ and $x_2$.  We will show that
it is possible to determine the leaf-to-leaf distances 
$r_\TT(x_i, x_j)$, $1 \le i,j \le n$.
As we recalled in the Introduction, this information
uniquely identifies the isomorphism type of $\TT$.

Again by assumption, the set 
$E := \{\ell > 0 : \bP\{W_4 = \ell \, | \, W_2 = \ell' + \ell''\} > 0\}$
has $\binom{n-2}{2}$ elements and  
$\bP\{W_4 = \ell \, | \, W_2 = \ell' + \ell''\} = \frac{1}{\binom{n-2}{2}}$
for each $\ell \in E$.  For a given $\ell \in E$ there is a unique
ordered pair $(x_i, x_j)$, $3 \le i \ne j \le n$, and a unique $e \ge 0$ such that
\[
\bP\{W_3 - W_2 = \ell_i, \, W_4 - W_3 = \ell_j - e 
\, | \, W_2 = \ell' + \ell'', \, W_4 = \ell\} > 0
\]
and
\[
\bP\{W_3 - W_2 = \ell_j, \, W_4 - W_3 = \ell_i - e 
\, | \, W_2 = \ell' + \ell'', \, W_4 = \ell\} > 0,
\]
in which case the two
conditional probabilities in question are both $\frac{1}{2}$.
Moreover, every ordered pair 
$(x_i, x_j)$, $3 \le i \ne j \le n$, corresponds to some unique $\ell \in E$ and $e \ge 0$ in this way.  
The event 
$\{W_2 = \ell' + \ell'', \, W_3 - W_2 = \ell_i, W_4 - W_3 = \ell_j - e, \, W_4 = \ell\}$
coincides with the event 
$\{\{Y_1,Y_2\} = \{x_1, x_2\}, \, Y_3 = x_i, \, Y_4 = x_j\}$
and the event 
$\{W_2 = \ell' + \ell'', \, W_3 - W_2 = \ell_j, W_4 - W_3 = \ell_i - e, \, W_4 = \ell\}$
coincides with the event $\{\{Y_1,Y_2\} = \{x_1, x_2\}, \, Y_3 = x_j, \, Y_4 = x_i\}$.
Considering the subtree of $\TT$ spanned by $\{x_1, x_2, x_3, x_4\}$
and ignoring the vertices with degree two to produce a simple tree, the leaves
$x_i$ and $x_j$ are siblings in this simple tree (as are $x_1$ and $x_2$),
and the quantity $e$ is the distance between the vertex in the subtree 
to which $x_i$ and $x_j$ are adjacent and the vertex to which 
$x_1$ and $x_2$ are adjacent; the lengths of the pendent edges
connecting $x_i$ and $x_j$ to the rest of the subtree are $\ell_i - e$
and $\ell_j - e$.
Thus, if the ordered pair
$(x_i, x_j)$ corresponds to $\ell \in E$
and $e \ge 0$, then, recalling the notation $r_\TT$ for the 
path length distance in $\TT$,
$r_\TT(x_1,x_2) = \ell'+\ell''$, 
$r_\TT(x_1, x_i) = \ell' + \ell_i$,
$r_\TT(x_1, x_j) = \ell' + \ell_j$,
$r_\TT(x_2, x_i) = \ell'' + \ell_i$,
$r_\TT(x_2, x_j) = \ell'' + \ell_j$,
and
$r_\TT(x_i, x_j) = \ell_i + \ell_j - e$.
 
Therefore, the joint probability distribution the random length
sequence $\cW_\TT$ uniquely determines the matrix of leaf-to-leaf distances 
in $\TT$ and hence the isomorphism type of $\TT$.

\end{proof}

\section{Ultrametric trees: Proof of \Cref{thrm:main ultrametric}}
\label{sec:ultrametric}

Recall that $\cJ_\TT$ is the set of sequences $(\ell_2, \ldots, \ell_n)$ 
such that $\bP\{W_k = \ell_k, \, 2 \le k \le n\} > 0$.
Write $\prec$ for the usual {\em lexicographic}
total order on $\cJ_\TT$ (that is $\ell' \prec \ell''$
if in the first coordinate where the two sequences differ the entry of
the $\ell'$ is smaller than the entry of $\ell''$).
Equivalently, $\ell' \prec \ell''$ if
either $\ell_2' < \ell_2''$ or $\ell_2' = \ell_2''$
and for the smallest $k \ge 2$ such that
$\ell_{k+1}' - \ell_k' \ne \ell_{k+1}'' - \ell_k''$ we have 
$\ell_{k+1}' - \ell_k' < \ell_{k+1}'' - \ell_k''$.
In this section we prove \Cref{thrm:main ultrametric} by showing that
that the tree $\TT$ is determined up to isomorphism by the
minimal element of $\cJ_\TT$.

We use a similar technique (but with a different total order)
to establish \Cref{thrm:main kary1} for 
 $k+1$-valent and rooted $k$-ary 
 combinatorial trees in \Cref{sec:kary trees}.

\begin{proof}
Let $(\ell_2, \ell_3, \ldots, \ell_n)$ be the minimal element of $\cJ_\TT$. 
Write $x_1, x_2, \ldots, x_n$ for an ordering of $\LL(\TT)$
such that $\ell_k = \WW_\TT(\{x_1, x_2, \ldots, x_k\})$  for
$k=2, \ldots, n$. 

We will establish by induction that for $2 \leq k \leq n$ the ultrametric
real tree spanned by the leaves 
$\{x_1, x_2, \ldots, x_k\}$ can be reconstructed from
$(\ell_2, \ell_3, \ldots, \ell_k)$
and, moreover, if we adopt the convention that
we draw ultrametric real trees in the plane with the root
at the top and leaves along the bottom, then
this particular real tree can be embedded in the plane with the leaves
$x_1, x_2, \ldots, x_k$ in order from left to right.

The claim is certainly true when $k=2$.  
Suppose the claim is true for $2,3,\ldots,k$. 
 
Write $\TT_k$ for the ultrametric real tree
spanned by $\{x_1, x_2, \ldots, x_k\}$ and denote the height of $\TT_k$
by $h_k$; 
that is, $h_k$ is the common distance from each of the leaves 
of $\TT_k$ to the root $\rho_k$ of $\TT_k$.  We can, of course,
suppose that $\TT_2 \subset \TT_3 \subset \ldots \subset \TT_n$.

If $\ell_{k+1}-\ell_k \ge h_k$,
then the ultrametric real tree 
$\TT_{k+1}$ spanned by $\{x_1, x_2, \ldots, x_k, x_{k+1}\}$
must consist of an arc of length 
$$
h_{k+1} = \frac{1}{2}(\ell_{k+1} - \ell_k + h_k)
$$
from the root $\rho_{k+1}$ of $\TT_{k+1}$ to the leaf
$x_{k+1}$ and an arc of length
$\frac{1}{2}(\ell_{k+1} - \ell_k - h_k)$ from 
``new root'' $\rho_{k+1}$
to the ``old root'' $\rho_k$.  
In this case we can, by the inductive hypothesis, 
certainly embed $\TT_{k+1}$ in
the plane with the leaf $x_{k+1}$ to the right of the leaves
$x_1, x_2, \ldots, x_k$.  

Assume, therefore, that
$\ell_{k+1}-\ell_k < h_k$.  Then the ultrametric real tree
$\TT_{k+1}$ must consist of $\TT_k$ and an arc 
of length $\ell_{k+1} - \ell_k$ joining
$x_{k+1}$ to a point $y \in \TT_k$.  
It will suffice to show that $y$ must be on the arc $[\rho_k, x_k]$ that
connects  $\rho_k$ to $x_k$ because
there is a unique
ultrametric real tree consisting 
of $\TT_k$ and an arc of length
$\ell_{k+1} - \ell_k$ joining a new leaf
to a point on the arc $[\rho_k, x_k]$ (this tree must have
root $\rho_k$ and the point where the arc of length $\ell_{k+1} - \ell_k$
attaches to $[\rho_k, x_k]$ must be at distance $h_k - (\ell_{k+1}-\ell_k)$
from $\rho_k$)
and, moreover, such a tree can be embedded in the plane with
the new leaf to the right of the leaves $\{x_1, x_2, \ldots, x_k\}$.

Suppose, then,  that $y$ is not on the arc $[\rho_k, x_k]$.  Let $j$ be the
maximum of the indices $i < k$ 
such that $y$ is on the arc connecting $x_i$ to $\rho_k$.
Write $u$ for the point that is closest
to $x_{j+1}$ in the subtree spanned by
$\{x_1, x_2, \ldots, x_j\}$ and $\rho_k$.  
Write $v$ for the point
that is closest to $x_{j+1}$
in the subtree spanned by $\{x_1, x_2, \ldots, x_j\}$. 
Equivalently, $v$ is the point
in the subtree spanned by $\{x_1, x_2, \ldots, x_j\}$
that is closest to $u$.  We may, of course, have $u=v$
(which occurs if and only if $h_{j+1} = h_j$).
By the inductive hypothesis, $u$ and $v$
are on the arc connecting $x_j$ to $\rho_k$
and 
$$
r_\TT(x_{j+1}, u) + r_\TT(u,v) = \ell_{j+1} - \ell_j.
$$

By construction, 
$y$ is the point closest
to $x_{k+1}$ in the subtree spanned by
$\{x_1, x_2, \ldots, x_j\}$ and $\rho_k$.  
Write $w$ for the point
closest to $x_{k+1}$ 
in the subtree spanned by $\{x_1, x_2, \ldots, x_j\}$.
Equivalently, $w$ is the point
in the subtree spanned by $\{x_1, x_2, \ldots, x_j\}$
that is closest to $y$.  We have
$$
\WW_\TT(\{x_1, \ldots, x_j, x_{k+1}\}) - \ell_j
= r_\TT(x_{k+1}, y) + r_\TT(y,w).
$$

By the definition of $j$, the points $y$ and $u$
are on the arc connecting $x_j$
to $\rho_k$ and
$r_\TT(u,x_j) > r_\TT(y, x_j)$.  This implies that
$r_\TT(u,v) \ge r_\TT(y,w)$.  It also implies,
by ultrametricity, that 
$$
r_\TT(x_{k+1},y) = r_\TT(x_j,y) <  r_\TT(x_j,u).
$$
Consequently, 
$$
\WW_\TT(\{x_1, \ldots, x_j, x_{k+1}\}) - \ell_j
< \ell_{j+1} - \ell_j.
$$
This, however,
contradicts the minimality of $(\ell_2, \ldots, \ell_n)$.
\end{proof}

\begin{rem} 
{As we noted in \Cref{rem:ultrametric determination}, it is interesting
to know whether it is possible to determine from the joint
probability distribution of the random length sequence whether
an edge-weighted tree is ultrametric.
The preceding proof of \Cref{thrm:main ultrametric} contains
a procedure for reconstructing $\TT$ from the minimal element of
$\cJ_\TT$ in the lexicographic order when $\TT$ is an ultrametric tree.
If $\TT$ is an arbitrary edge-weighted tree and this procedure is applied
to the minimal element of $\cJ_\TT$ in the lexicographic order, then
it will still produce an ultrametric tree and so a necessary
condition for $\TT$ to be ultrametric is that the joint
probability distribution of the random length sequence of this
ultrametric tree coincides with the joint
probability distribution of $\cW_\TT$.  

Along the same lines,
suppose that $\TT$ is an arbitrary edge-weighted tree and, thinking
of $\TT$ as a real tree, we root it at the unique point $\rho$ such that 
\[
\max_{v \in \LL(\TT)} r_\TT(\rho,v) 
=
r^*
:=
\frac{1}{2}\max_{u \in \LL(\TT)} \max_{v \in \LL(\TT)} r_\TT(u,v).
\]
Then $\rho$ will have $k$ children for some $k$.  Let
$m_i$, $1 \le i \le k$, be the number of leaves $v$ in the
subtree below the $i^{\mathrm{th}}$ child of $\rho$
such that $r_\TT(\rho,v) = r^*$.  It is clear that
$\TT$ is ultrametric if and only if 
$m_1 + \cdots + m_k = n$.  Let $n_1, \ldots, n_\ell$
be a listing of the nonzero terms in the list
$m_1, \ldots, m_k$.  Note for $2 \le j \le \ell$ that
\[
\bP\{W_2 = 2 r^*, \ldots, W_j = j r^*\} 
= 
j ! \frac{1}{n(n-1) \cdots (n-j+1)}
\sum_{1 \le h_1 < \ldots < h_j \le \ell}
n_{h_1} \cdots n_{h_j}
\]
and
\[
\max\{j \ge 2 : \bP\{W_2 = 2 r^*, \ldots, W_j = j r^*\} > 0\}
=
\ell.
\]
Thus, the joint probability distribution of $\cW_\TT$
determines $\ell$ and the values of the elementary symmetric polynomials
of degrees $2 \le j \le \ell$ evaluated at $n_1, \ldots, n_\ell$, and
we want to know whether $n_1 + \cdots + n_\ell$, the value of the
elementary symmetric polynomial of degree $1$ evaluated at $n_1, \ldots, n_\ell$,
is $n$.  The elementary symmetric polynomials of degrees $1,2,\ldots,\ell$
in $\ell$ real variables are algebraically independent over the reals,
and so we cannot expect to recover $n_1 + \cdots + n_\ell$ from
the values of the other elementary symmetric polynomials.  However,
there are inequalities connecting the values of the various elementary
symmetric polynomials that can be used to establish necessary conditions
and sufficient conditions for $\TT$ to be ultrametric.  For example, set
\[
p_1 := \frac{1}{\ell}(n_1 + \cdots + n_\ell)
\]
and
\[
\begin{split}
p_j 
& := 
\frac{1}{\binom{\ell}{j}} 
\sum_{1 \le h_1 < \ldots < h_j \le \ell}
n_{h_1} \cdots n_{h_j} \\
& =
\frac{1}{\binom{\ell}{j}} \frac{n(n-1) \cdots (n-j+1)}{j!}
\bP\{W_2 = 2 r^*, \ldots, W_j = j r^*\}, \quad 2 \le j \le \ell. \\
\end{split}
\]
If $\alpha_1, \ldots, \alpha_\ell$ and $\beta_1, \ldots, \beta_\ell$
are positive constants such that
\[
\alpha_1 + 2 \alpha_2 + \cdots + \ell \alpha_\ell
=
\beta_1 + 2 \beta_2 + \cdots + \ell \beta_\ell
\]
and
\begin{equation}
\label{eq:alpha beta}
\alpha_j + 2 \alpha_{j+1} + \cdots + (\ell-j+1) \alpha_\ell
\ge
\beta_j + 2 \beta_{j+1} + \cdots + (\ell-j+1) \beta_\ell, \quad
2 \le j \le \ell,
\end{equation} 
then, by \cite[Theorem 77, Chapter II]{MR944909}
\[
\prod_{j=1}^\ell p_j^{\alpha_j} \le \prod_{j=1}^\ell p_j^{\beta_j}.
\]
Thus, if $\alpha_2, \ldots, \alpha_\ell$
and $\beta_2, \ldots, \beta_\ell$ satisfy the inequalities
\eqref{eq:alpha beta} and 
\[
\gamma := \sum_{j=2}^\ell j (\beta_j - \alpha_j),
\] 
then
\[
p_1 
\le 
\left(
\prod_{j=2}^\ell p_j^{\beta_j - \alpha_j}
\right)^{\frac{1}{\gamma}}
\]
when $\gamma > 0$,
and the opposite inequality hold
when $\gamma < 0$.  This observation leads to necessary
conditions and sufficient conditions for $\TT$ to be ultrametric.}
\label{rem:more ultrametric determination}
\end{rem}

\begin{rem}
As we will see in \Cref{sec:kary} for 
$k+1$-valent and rooted $k$-ary combinatorial trees, 
a somewhat similar proof argument based on the consideration of length sequences that are minimal with respect
to a suitable order leads to a stronger result in that case. There we can not only determine $\TT$ from the joint probability
distribution of its random length sequence, but if we have a random tree $\cT$ with a fixed number of leaves, then it is possible to determine the distribution of $\cT$ from the joint probability distribution of the 
random length sequence obtained by first picking
a realization of $\cT$ and then independently picking a random ordering of the leaves to build a random length sequence.
 
Formally, we have some space $\bT$ of isomorphism types of trees, 
a corresponding space $\bS$ of possible length sequences, and a probability kernel
$\mu$ from $\bT$ to $\bS$, where, for $\TT \in \bT$,  $\nu(\TT, \cdot)$ is the element of $\cP(\bS)$, 
the space of probability measures on
$\bS$, that is the joint probability distribution of the random length sequence built from $\TT$.  
An affirmative answer to \Cref{q:main}  for a particular $\bT$ means that the map $\TT \mapsto \nu(\TT, \cdot)$ 
from $\bT$ to $\cP(\bS)$ is injective.
Given an element $\mu$ of $\cP(\bT)$, the space of probability measures on $\bT$, let $\mu \nu \in \cP(\bS)$
be defined as usual by $\mu \nu(B) = \int_\bT \nu(\TT,B) \, \mu(d\TT)$ for $B \subseteq \cS$.  The stronger results obtained in
\Cref{sec:kary} say that, in the situations considered there, the map $\mu \mapsto \mu \nu$ 
from $\cP(\bT)$ to $\cP(\bS)$ is injective.

One can ask if an analogous strengthening is also true for ultrametric trees.
A proof along the lines of that given for \Cref{thrm:main kary2} 
doesn't appear to apply immediately in this situation where the relevant space $\bT$ is uncountable
rather than finite.  We leave this as one of many open questions.
\end{rem}

\section{Caterpillar trees: Proof of \Cref{thrm:main caterpillar}}
\label{sec:caterpillar}

Recall that a caterpillar is a (not necessarily simple)
combinatorial tree such that deleting the leaves of the
tree results in a path consisting of $\ell+1$ vertices 
(and hence $\ell$ edges of length $1$).  

\begin{rem}
Choosing one end of the path,
we can label the vertices on path consecutively with $0,1,\ldots,\ell$ 
and denote by $n_r$ the leaves that are attached to vertex $r$ on the
path. Both $n_0$ and $n_\ell$ are non-zero, but the remaining $n_i$ may
be zero.

The isomorphism types of caterpillars with $n$ leaves 
are thus seen to be in a bijective correspondence 
with equivalence classes of nonnegative integer sequences 
$(n_0, n_1, \ldots, n_{\ell-1}, n_\ell)$, where
$n = n_0 + \cdots + n_\ell$
and
$n_0, n_{\ell} \ne 0$,
and we declare that 
$(n_0, n_1, \ldots, n_{\ell-1}, n_\ell)$
and
$(n_\ell, n_{\ell-1}, \ldots, n_1, n_0)$
are equivalent.
\end{rem}

The proof of the following, which establishes the
first claim in \Cref{thrm:main caterpillar}, is straightforward and we omit it.

\begin{prop}{A combinatorial tree $\TT$ with $n$ leaves is a caterpillar with an associated path of length $\ell$ if and only if
$$
\max \{ k : \Prb \{ W_2 = k + 2 \} > 0\} = \ell
$$
and $W_n = \ell + n$ almost surely.
}
\label{prop:det if caterpillar}
\end{prop}

We now turn to the proof of the main claim in \Cref{thrm:main caterpillar}.

\begin{proof}
Consider a box with $n$ tickets.  Each ticket has a label belonging to $\{0,1,\ldots,\ell\}$
and there are $n_i$ tickets with label $i$ for $0 \le i \le \ell$.  Let $X_1, X_2, \ldots, X_n$
be the result of drawing tickets uniformly at random from the box without replacement
and noting their labels.
Set 
$$
K_r := \max_{1 \le j \le r} X_j - \min_{1 \le j \le r} X_j.
$$ 
It is clear that $(W_2, W_3, \ldots, W_n)$ 
has the same joint probability distribution as $(K_2 + 3, K_3 + 3, \ldots, K_n+n)$, 
and so it suffices to show that it is possible to determine
$\{(n_0, n_1, \ldots, n_{\ell-1}, n_\ell), (n_\ell, n_{\ell-1}, \ldots, n_1, n_0)\}$
from a knowledge of the joint probability distribution of $\cK := (K_2, \ldots, K_n)$
(that is, it is possible to determine up to a reflection
the vector that gives the number of tickets with each label).

To begin with, note that, as in \Cref{prop:det if caterpillar},
$$
\max\{k : \bP\{K_2 = k\} > 0\} = \ell,
$$
and so we can determine $\ell$ from the joint probability distribution of $\cK$.  

Observe next that
\[
\begin{split}
\bP\{K_2 = \ell\} 
& = \bP\{(X_1, X_2) \in \{(0,\ell), (\ell,0)\}\} \\
& = 2\frac{n_0 n_\ell}{n(n-1)}, \\
\end{split}
\]
and 
\[
\max\{k : \bP\{K_2 = 0, \ldots, K_k = 0, \, K_{k+1} = \ell\} > 0\}
= n_0 \vee n_\ell.
\]
We can thus determine the multiset $\{n_0, n_\ell\}$
and, in particular, $n_0 + n_\ell$.

For $1 \le r < \frac{\ell}{2}$ we have
\[
\begin{split}
& \bP\{K_2 = r, \, K_3 = \ell \} \\
& \quad =
\bP\{(X_1, X_2, X_3) \in 
\{(0,r,\ell), (r,0,\ell), (\ell,\ell-r,0), (\ell-r,\ell,0)\}
\} \\
& \quad = 
\frac{2 n_0(n_r+n_{\ell-r})n_\ell}{n(n-1)(n-2)}, \\
\end{split}
\]
and so we can determine $n_r+n_{\ell-r}$. If $\ell$ is even,
then
\[
\begin{split}
& \bP\{K_2 = \frac{\ell}{2}, \, K_3 = \ell\} \\
& \quad =
\bP\left\{(X_1, X_2, X_3) \in 
\left\{\left(0,\frac{\ell}{2},\ell\right), \left(\frac{\ell}{2},0,\ell\right), 
\left(\ell,\frac{\ell}{2},0\right), \left(\frac{\ell}{2},\ell,0\right)\right\}
\right\} \\
& \quad = 
\frac{4 n_0 n_{\frac{\ell}{2}} n_\ell}{n(n-1)(n-2)}, \\
\end{split}
\]
and so we can determine $n_{\frac{\ell}{2}}$.

Also,  
\[
\begin{split}
\bP\{K_2 = 0\} 
& = \sum_{i=0}^\ell \bP\{X_1 = r, X_2 = r\} \\ 
& = \frac{\sum_{r=0}^\ell n_r(n_r-1)}{n(n-1)} \\
& = \frac{\sum_{r=0}^\ell n_r^2 - n} {n(n-1)} \\
\end{split}
\]
and, for $1 \le k \le \ell$,
\[
\begin{split}
\bP\{K_2 = k\}
& = \sum_{r=0}^{\ell-k} \bP\{(X_1, X_2) \in \{(r,r+k), (r+k,r)\}\} \\
& = \frac{2 \sum_{r=0}^{\ell-k} n_r n_{r+k}}{n(n-1)}. \\
\end{split}
\]
We can therefore determine $\sum_{r=0}^{\ell-k} n_r n_{r+k}$ for 
$0 \le k \le \ell$. 

We claim that we the information we have just derived suffices to determine  
$\{(n_0, n_1, \ldots, n_{\ell-1}, n_\ell), (n_\ell, n_{\ell-1}, \ldots, n_1, n_0)\}$.
That is, if $n_0', \ldots, n_\ell'$ is a sequence with 
\[
n_0 + \cdots + n_\ell = n_0' + \cdots + n_\ell' = n,
\]
\[
n_r+ n_{\ell-r} = n_r'+ n_{\ell-r}'
\]
for $0 \le r \le \ell$,
and
\[
\sum_{r=0}^{\ell-k} n_r n_{r+k} = \sum_{r=0}^{\ell-k} n_r' n_{r+k}'
\]
for $0 \le k \le \ell$,  then either 
$n_r = n_r'$ for $0 \le r \le \ell$ 
or $n_r = n_{\ell-r}'$ for $0 \le r \le \ell$.

To see that this is so, introduce the Fourier transforms
\[
g(z) := \sum_{k=0}^\ell n_k e^{i z k}
\]
and
\[
g'(z) := \sum_{k=0}^\ell n_k' e^{i z k}
\]
for $z \in \bC$.  These are entire functions that uniquely determine
$n_0, \ldots, n_\ell$ and $n_0', \ldots, n_\ell'$.  Note that
\[
\sum_{k=0}^\ell n_{\ell-k} e^{i z k}
= e^{i z \ell} g(-z),
\]
and a similar formula holds for $g'$.
It will thus suffice to show that either $g(z) = g'(z)$ or
$g(z) = e^{i z \ell} g'(-z)$ (equivalently, 
$g'(z) = e^{i z \ell} g(-z)$).  

It follows from the assumption that
\[
\sum_{r=0}^{\ell-k} n_r n_{r+k} = \sum_{r=0}^{\ell-k} n_r' n_{r+k}'
\]
for $0 \le k \le \ell$
that if we define $N: \bZ \to \bZ$ by
\[
N(j) = 
\begin{cases} 
n_j,& \quad 0 \le j \le \ell, \\
0,& \quad \text{otherwise,}
\end{cases}
\]
and define $N'$ similarly,
then
\[
\sum_{\{r,j \in \bZ: r-j = k\}} N(r) N(j)
=
\sum_{\{r,j \in \bZ: r-j = k\}} N'(r) N'(j)
\]
for all $k \in \bZ$ and hence
$$
g(z) g(-z) = g'(z) g'(-z)
$$
for all $z \in \bC$.  By Theorem 2.2 in \cite{rosenblatt1982structure}, 
 there exist finitely supported functions $C: \bZ \to \bZ$ and 
 $D: \bZ \to \bZ$ such that if we set
\[
\phi(z) := \sum_{k \in \bZ} C(k) e^{i z k}
\]
and
\[
\psi(z) := \sum_{k \in \bZ} D(k) e^{i z k},
\]
then
\[
g(z) = \phi(z) \psi(z)
\]
and
\[
g'(z) = \phi(z) \psi(-z).
\]

It follows from the assumption that
\[
n_r+ n_{\ell-r} = n_r'+ n_{\ell-r}'
\]
for $0 \le r \le \ell$
that
\[
g(z) + e^{i z \ell} g(-z)
=
g'(z) + e^{i z \ell} g'(-z)
\]
for all $z \in \bC$.
Therefore,
\[
\phi(z) \psi(z) + e^{i z \ell} \phi(-z) \psi(-z)
=
\phi(z) \psi(-z) + e^{i z \ell} \phi(-z) \psi(z)
\]
and hence
\[
(\phi(z) - e^{i z \ell} \phi(-z))(\psi(z) - \psi(-z)) = 0
\]
for all $z \in \bC$.  Because the functions 
$z \mapsto \phi(z) - e^{i z \ell} \phi(-z)$ and
$z \mapsto \psi(z) - \psi(-z)$ are both entire, we must have either
that $\phi(z) = e^{i z \ell} \phi(-z)$ for all $z \in \bC$
or $\psi(z) = \psi(-z)$ for all $z \in \bC$.
If $\phi(z) = e^{i z \ell} \phi(-z)$ for all $z \in \bC$, then
\[
g'(z) = \phi(z) \psi(-z) 
= e^{i z \ell} \phi(-z) \psi(-z) = e^{i z \ell} g(-z)
\]
and
$n_i = n_{\ell-i}'$ for $0 \le i \le \ell$.
If $\psi(z) = \psi(-z)$ for all $z \in \bC$, then
\[
g'(z) = \phi(z) \psi(-z) 
= \phi(z) \psi(z) = g(z)
\]
and 
$n_r = n_r'$ for $0 \le r \le \ell$.
\end{proof}


\section{$k+1$-valent and rooted $k$-ary trees}
\label{sec:kary trees}

We now turn our focus to the cases of $(k+1)$-valent and $k$-ary trees. 
Recall that
a {\em $(k+1)$-valent tree} is a tree with all vertices of degree either $k+1$ or $1$. 
For $k \geq 2$ a {\em rooted $k$-ary tree} is a tree with one vertex of degree $k$ 
and the rest of degrees either $k+ 1$ or $1$. We refer to the rooted $2$-ary tree 
as a {\em rooted binary tree}. Note that any $k$-ary tree is obtained by removing
one leaf of a suitable $(k+1)$-valent trees.

Our general proof methodology for these families of trees is similar to that used in 
\Cref{sec:ultrametric} for ultrametric trees. We first define a particular 
class of sequences that can appear as elements of $\cJ_\TT$ 
(the down-split sequences) and a total order on such sequences. 
We then show that the minimal down-split sequence in  $\cJ_{\TT}$ 
uniquely identifies $\TT$.

The idea of the proof is the same for all $k$
and depends on the following fact.

\begin{lemma}{Let $\TT$ be a $(k+1)$-valent tree or a rooted $k$-ary tree and let
$\SS$ be a subtree of $\TT$. Then $\SS$ is a rooted $k$-ary tree if and only if
$$
\# \EE(\SS) = \frac{k}{k - 1}( \# \LL(\SS) - 1).
$$}
\label{lem:k subtree size}
\end{lemma}

\begin{proof}
Because $\SS$ is a subtree of $\TT$, 
every interior vertex of $\SS$ has degree at most $k+1$. 
Write $d_1 := \# \LL(\SS), d_2, \ldots, d_{k+1}$ for the number of vertices of $\SS$ 
of degrees $1, 2, \ldots, k+1$. We need to show that $d_j=0$ for $1 < j \le k-1$ and
$d_k=1$, or, equivalently, that $d_k=1$ and 
$d_{k+1} = \sum_{j=2}^{k+1} d_j - 1
= \# \VV(\SS) - d_1 - 1 
= \# \EE(\SS) - d_1$.
This is in turn equivalent to showing that
\[
\sum_{j=2}^{k+1} j d_j = k + (k+1) (\# \EE(\SS) - d_1),
\]
which, by the ``handshaking identity''
\[
2 \# \EE(\SS) = \sum_{j=1}^{k+1} j d_j,
\]
becomes
\[
2 \# \EE(\SS) - d_1 = k + (k+1) (\# \EE(\SS) - d_1)
\]
or, upon rearranging,
\[
\# \EE(\SS) = \frac{k}{k-1} (d_1 - 1) = \frac{k}{k-1} (\# \LL(\SS) - 1).
\]
\end{proof} 

For simplicity of notation we present the details
of the proof for the case of (unrooted) $3$-valent trees
and rooted binary trees
(that is, $k=2$). 
We end in \Cref{sec:kary} with a discussion of the extension 
to general $k$.

\subsection{$3$-valent and rooted binary trees}
\label{sec:binarytrees}
%
%
Our proof of \Cref{thrm:main kary1} begins with an analysis of random length sequences 
for marked (also known as planted) $3$-valent trees. A {\em marked $3$-valent tree} 
$(\TT, v)$ is an $3$-valent tree $\TT$ and a distinguished leaf $v$ of $\TT$. 
We define the {\em modified random length sequence} $\cW_{(\TT, v)}$ of $(\TT, v)$ 
to be the random length sequence $\cW_{\TT}$ of $\TT$ conditioned on $Y_1 = v$.

\subsection{Down-split sequences}

We need to distinguish some particular sequences that
appear in the support of $\cW_{(\TT,v)}$.

\begin{rem}
{As usual, we can define a partial order on $\VV(\TT)$ by declaring
that $x$ precedes $y$ if $x \ne y$ is on the path between $v$ and $y$, and we can extend this
partial order to a total order $<$ such that if $w,x,y,z$ are such that
$w$ and $x$ are not comparable in the partial order  but $w < x$, $w$ precedes $y$
in the partial order,
and $x$ precedes $z$ in the partial order, 
then $y < z$.  Such a total order corresponds to embedding $\TT$
in the plane and listing the elements of $\VV(\TT)$ in the order
they are encountered as one walks around  $\TT$ starting from $v$.

Suppose that $v=y_1 < y_2 < \ldots < y_n$ is the ordered
listing of $\LL(\TT)$.  Set
$s_k = \WW_\TT(\{y_1, \ldots, y_k\})$, $2 \le k \le n$.
If $s_k = 2k-2$, then the subtree spanned by the $k$ leaves
$\{y_1, \ldots, y_k\}$ has $2k-2$ edges and hence, by \Cref{lem:k subtree size},
this subtree is a binary tree.  If we write $o$ for the vertex adjacent to the marked
leaf $v$,  denote by $v',v''$ the other two vertices adjacent to $o$, and suppose that
$v' < v''$, then it must be the case that 
$\{y_2, \ldots, y_{k_s}\} = \{y \in \LL(\TT) : v' \le y\}$.  Write $\TT'$ 
(respectively, $\TT''$) for
the subtree of $\TT$ consisting of $w$ and the vertices $u$ such that $v'$ 
(respectively, $v''$) is on the path from $o$ to $u$.  The sequence
$(s_2', \ldots, s_{n'}'):= (s_2-1, \ldots, s_{k_s}-1)$ satisfies
$s_k' = \WW_{\TT'}(o, y_2, \ldots, y_k)$ for $2 \le k \le n' = k_s$.
The sequence
$(s_2'', \ldots, s_{n''}'' := (s_{k_s + 1} - (2k_s - 2), \ldots, s_{n} - (2k_s - 2) )$
satisfies 
$s_k'' = \WW_{\TT''}(o, y_{k_s + 1}, \ldots, y_{k_s+ k - 1})$ for 
$2 \le k \le n'' = n - k_s + 1$.}
\label{rem:down-split construction}
\end{rem}

\begin{defn}
\label{def:downsplit}
A {\em down-split sequence} is an element of the class of increasing sequences of positive
integers defined recursively as follows.
The sequence 
$$
s = (1)
$$ 
is a down-split sequence.

A sequence $s = (s_2, \ldots, s_n)$, $n > 2$, is down-split if 
\[
\{ 2 \leq k < n \colon s_k = 2k - 2 \} \ne \emptyset
\]
and, setting
$$
k_s = \inf \{ 2 \leq k < n \colon s_k = 2k - 2 \},
$$
\begin{itemize}
\item $(s_2 - 1, \ldots, s_{k_s} - 1)$ is down-split,
\item $(s_{k_s + 1} - (2k_s - 2), \ldots, s_{n} - (2k_s - 2) )$ is down-split.
\end{itemize}
The index $k_s$ is the {\em splitting index} of $s$. 
\end{defn}

\begin{ex}
{For $n=3$, the sequence $s = (s_2,s_3) = (2,3)$ is a down-split sequence.  Here $k_s = 2$,
$(s_2 - 1, \ldots, s_{k_s} - 1) = (1)$ and 
$(s_{k_s + 1} - (2k_s - 2), \ldots, s_{n} - (2k_s - 2) ) = (1)$.}
\end{ex}

The following result is immediate from \Cref{rem:down-split construction}.

\begin{lemma}{For every marked $3$-valent tree $(\TT,v)$ there is at least one down-split sequence $s$ with
$$
\Prb \{ \cW_{(\TT,v)} = s \} > 0.
$$}
\label{lem:exist of split}
\end{lemma}

We record the following fact for later use.

\begin{lemma}{If $s = (s_2,\ldots, s_n)$ is a down-split sequence then 
$s_n = 2n - 3$.}
\label{lem:size of downsplit}
\begin{proof}
This follows easily by induction. If $s$ splits at $k_s$, then, as 
$$
(s_2', \ldots, s_{n'}')
=
(s_{k_s + 1} - (2k_s - 2), \ldots, s_n - (2k_s - 2))
$$
is a down-split sequence with $n' = n - k_2 + 1$, we have by the inductive hypothesis that
$$
s_{n} - (2k_s - 2) = 2( n - k_s + 1) - 3
$$ 
and the claim follows.
\end{proof}
\end{lemma}

\begin{ex}
\label{ex:nonunique split seq}
Given any down-split sequence $s$, it is possible to reverse the argument in \Cref{rem:down-split construction}
 and construct a marked $3$-valent tree with
a suitable total ordering on its vertices such that $s$ is the corresponding down-split sequence.
However, a marked $3$-valent tree $(\TT,v)$ is not uniquely identified by an arbitrary 
down-split sequence in the support of $\cW_{(\TT,v)}$, as the example in 
\Cref{fig:DownSplitCex} shows.

\begin{figure}[ht]
    \centering
    \includegraphics[width=0.4\textwidth]{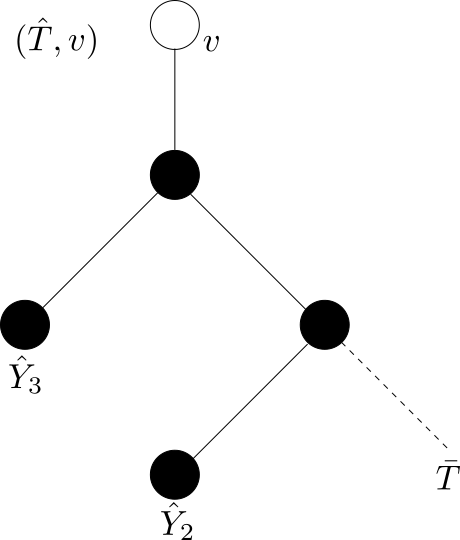}
    \hspace{5 pc}
    \includegraphics[width=0.4\textwidth]{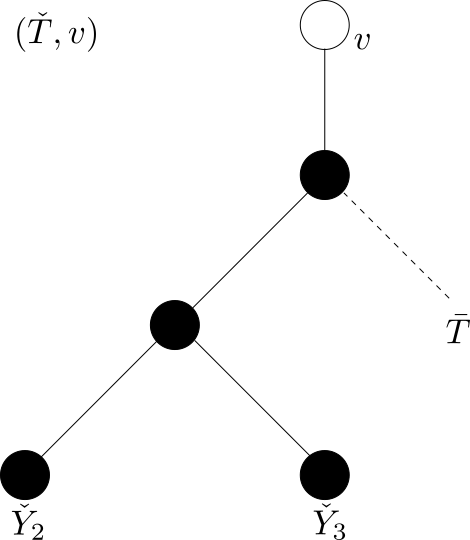}
    \caption{Two marked binary trees $(\hat \TT,v)$ and $(\check \TT,v)$ with particular realizations of the random selection of leaves.}
    \label{fig:DownSplitCex}
\end{figure}

Write $(\hat Y_1, \ldots, \hat Y_n)$ and $(\check Y_1, \ldots, \check Y_n)$ 
for the random selections of the leaves of $\hat \TT$ and $\check \TT$.
Suppose that the realizations are such that $\hat Y_k = \check Y_k \in \bar \TT$ for $4 \le k \le n$  
and that these leaves of the subtree $\bar \TT$ appear 
in an order of the type discussed in \Cref{rem:down-split construction}. The corresponding
realizations for the modified random length sequences are equal.  The common value $(3,4,\ldots)$ is a down-split sequence with
splitting index $3$.  Thus, two non-isomorphic marked $3$-valent trees can have a common down-split sequence in the
supports of their modified random length sequences.  Note that the common down-split sequence results from taking the leaves
of $\check \TT$ according to an order of the type described in
\Cref{rem:down-split construction}, but this is not the case for $\hat \TT$.
\end{ex}

With \Cref{ex:nonunique split seq} in mind we see that it would be useful to have a way of recognizing down-split
sequences in the support of $\cW_{(\TT,v)}$ that result from realizations where the leaves are selected in an order that arises
from a suitable total order on the vertices of $\TT$.  The key is the following total order on down-split sequences.
We re-use the notation $\prec$ that was used in \Cref{sec:ultrametric} for the lexicographic order.

\begin{defn}
Define a total order $\prec$ on the set of down-split sequences of a given length
recursively as follows. Firstly, $(1) \prec (1)$ does not hold.
Next, let $s, r$ be down-split sequences indexed by $\{2, \ldots, n\}$ 
with respective splitting indices $k_s$ and $k_r$.  Set
$$
s' = (s_2 - 1, \ldots, s_{k_s} - 1 ), 
\quad
r' = ( r_2 - 1, \ldots, r_{k_r} - 1),
$$
and
$$
s'' = (s_{k_s + 1} - (2k_s - 2), \ldots, s_n - (2k_{s} - 2) ), 
\quad
r'' = (r_{k_r + 1} - (2k_r - 2), \ldots, r_n - (2k_{r} - 2) ).
$$
Declare that
$$
s \prec r
$$
if 
\[
k_s < k_r
\] 
or 
\[
k_s = k_r \quad \text{and} \quad s' \prec r'
\]
or
\[
k_s = k_r \quad \text{and} \quad s' = r' \quad \text{and}  \quad s'' \prec r''.
\]
\end{defn}

The next result follows easily by induction.

\begin{lemma}{The binary relation $\prec$ is a total order on the set of 
down-split sequences of a given length.}
\end{lemma}

\begin{defn}
The {\em minimal down-split sequence} for a marked $3$-valent tree $(\TT,v)$ is the minimal element 
(with respect to the total order $\prec$) of the set
$$
\{ s \text{ down-split }\colon \Prb \{ \cW_{(\TT,v)} = s \} > 0 \}.
$$
\end{defn}


We now proceed to establish some results that culminate
in showing that $(\TT,v)$ is determined by its minimal down-split sequence.

\begin{lemma}{Let $(\TT,v)$ be a marked $3$-valent tree, with modified random length sequence
$\cW_{(\TT, v)} = (W_2, \ldots, W_n)$  constructed from the random sequence of leaves
$(Y_1, \ldots, Y_n)$ with $Y_1 = v$.  
Denote by $o$ the vertex adjacent to the marked
leaf $v$ and denote by $v',v''$ the other two vertices adjacent to $o$.
Write $\TT'$ 
(respectively, $\TT''$) for
the subtree of $\TT$ consisting of $o$ and the vertices $u$ such that $v'$ 
(respectively, $v''$) is on the path from $o$ to $u$.
Set
$$
m := \inf\{ k  \colon \Prb \{ W_k = 2k - 2 \} > 0 \}.
$$
Then $W_m = 2m - 2$ if and only if
$$
Y_2, \ldots, Y_m \in \LL(\TT') \text{ and } Y_{m + 1}, \ldots, Y_n \in \LL(\TT''),
$$
or {\em vice versa}.}
\label{lem:split subtrees}
\end{lemma}
\begin{proof}
If $Y_2, \ldots, Y_m \in \LL(\TT')$ and $Y_{m + 1}, \ldots, Y_n \in \LL(\TT'')$, then the subtree spanned by $\{v,Y_2, \ldots, Y_m\}$ 
consists of the leaf $v$ adjoined to $\TT'$ via an edge to the vertex $o$. This subtree is a rooted binary tree with root $o$. 
It follows from \Cref{lem:k subtree size} that $W_m = 2m - 2$.

For the other direction, assume that $W_m = 2m -2$. By \Cref{lem:k subtree size}, the subtree $\SS$ spanned by $\{v, Y_2, \ldots, Y_m\}$ is a rooted binary tree with $m$ leaves. We have $\LL(\SS) \subseteq \LL(\TT)$, $v \in \LL(\SS)$, and 
$\LL(\SS) \setminus \{v\} \subseteq (\LL(\TT') \setminus \{o\}) \cup (\LL(\TT'') \setminus \{o\})$.
We need to show that $\SS$ consists of the leaf $v$ adjoined to either $\TT'$ or $\TT''$ via an edge to the vertex $o$
that is common to both $\TT'$ and $\TT''$. 

By the construction prior to the statement of \Cref{lem:exist of split} we know that 
\[
m \leq \# \LL(\TT') \wedge \# \LL(\TT'')
\]
and so if $\LL(\TT') \setminus \{o\} \subseteq \LL(\SS) \setminus \{v\}$, then 
$\LL(\TT'') \cap \LL(\SS) = \LL(\TT'') \setminus \{o\} \cap \LL(\SS) \setminus \{v\} = \emptyset$ and similarly with the roles of $\TT'$
and $\TT''$ reversed. 

We can rule out the possibility that $\LL(\SS)$ intersects both 
$\LL(\TT')$ and $\LL(\TT'')$ as follows. If $\LL(\TT') \cap \LL(\SS) \ne \emptyset$
and $\LL(\TT'') \cap \LL(\SS) \ne \emptyset$, then $\LL(\TT') \cap \LL(\SS)$ must be a proper subset of $\LL(\TT') \setminus \{o\}$
and $\LL(\TT'') \cap \LL(\SS)$ must be a proper subset of $\LL(\TT'') \setminus \{o\}$.
 If $\LL(\TT') \cap \LL(\SS)$ is a proper, nonempty subset of $\LL(\TT') \setminus \{o\}$, then $\SS$ must have a degree $2$ vertex that belongs to $\VV(\TT') \setminus \{o\}$, and similarly for $\TT''$. However, $\SS$ is a rooted binary tree and cannot have
two or more vertices of degree $2$.

Finally, we need to rule out the possibility of $\LL(\SS) \setminus \{v\}$ is a proper subset of $\LL(\TT') \setminus \{o\}$ or 
$\LL(\TT'') \setminus \{o\}$. However, if $\LL(\SS) \setminus \{v\}$ is a proper subset of $\LL(\TT') \setminus \{o\}$, then $\SS$ would have at least one degree $2$ vertex in that belongs to $\VV(\TT') \setminus \{o\}$ as well as the degree $2$ vertex $o$, which contradicts
$s$ being a rooted binary tree. The same argument holds with $\TT''$ in place of $\TT'$.
\end{proof}

\begin{cor}{Let $(\TT,v)$ be a marked $3$-valent tree
with modified random length sequence $\cW_{(\TT, v)} = (W_2, \ldots, W_n)$.
Then
\[
m:= \inf\{ k  \colon \Prb \{ W_k = 2k - 2 \} > 0 \}
\]
is the splitting index for the minimal down-split sequence for $(\TT,v)$.}
\label{cor:split subtrees}
\end{cor}
\begin{proof}
If $k_s$ is the splitting index of any down-split sequence $s$ in the support of $\cW_{(\TT,v)}$, then 
$s_{k_s} = 2k_s - 2$ by definition. Thus
$\Prb \{ W_{k_s} = 2k_s - 2 \} > 0$ and hence $m \le k_s$.

On the other hand, let $o,v',v'', \TT',\TT''$ 
be as in the statement of \Cref{lem:split subtrees}.  It follows from that result that
$m = \# \LL(\TT') \wedge \# \LL(\TT'')$.  By the construction in \Cref{rem:down-split construction}
if $m = \# \LL(\TT')$ or the analogous one with the roles of $\TT'$ and $\TT''$ reversed if $m = \# \LL(\TT'')$,
we may construct a down-split sequence for $(\TT,v)$ that has splitting index $m$.  By the definition of
the total order $\prec$, the splitting index for the minimal down-split sequence for $(\TT,v)$ is at most $m$.
\end{proof}

\begin{prop}{Let $s$ be the minimal down-split sequence for a marked $3$-valent tree $(\TT,v)$. 
There is no other marked $3$-valent tree for which $s$ is the minimal down-split sequence.}
\label{prop:det by downsplit}
\end{prop}
\begin{proof}
We will prove this by induction. The claim is clearly true for the down-split sequence $s = (1)$.

Let $(\TT, v)$ be a marked $3$-valent tree and $s$ the minimal down-split sequence for $(\TT,v)$. Define $o,v',v'', \TT',\TT''$ 
as in the statement of \Cref{lem:split subtrees}.  Let $k_s$ be the splitting index of $s$.  Let $y_1, \ldots, y_n$
be an ordered listing of $\LL(\TT)$ such that $\WW_{(\TT,v)}(\{y_1, \ldots, y_k\}) = s_k$ for $2 \le k \le n$.
By \Cref{cor:split subtrees} and \Cref{lem:split subtrees} we must either have
$\{y_2, \ldots, y_{k_s}\} = \LL(\TT') \setminus \{o\}$ and $\{y_{k_s + 1}, \ldots, y_{n}\} = \LL(\TT'') \setminus \{o\}$
or the analogous conclusion with the roles of $\TT'$ and $\TT''$ interchanged holds (if $\# \LL(\TT') \ne \LL(\TT'')$, then only
one alternative is possible).  We may suppose without loss of generality that the choice of $v'$ and $v''$ is such that
the first alternative holds.

Set
$$
s' := (s_2 - 1, \ldots, s_{k_s} - 1), \quad s'' := (s_{k_s + 1} - (2k_s - 2), \ldots, s_{n} - (2k_s - 2) ).
$$ 
By definition, $s'$ and $s''$ are down-split sequences.
Because
$\Prb \{ \cW_{(\TT, v)} = s \} > 0$,
we have
$$
\Prb\{\cW_{(\TT', o)} = s' \} > 0
$$
and
$$
\Prb\{\cW_{(\TT'', o)} = s''\} > 0.
$$

We claim that $s'$ must be the minimal down-split sequence for $(\TT', o)$. 
To see this, note that if there was a down-split sequence $\tilde{s}'$ with $\tilde{s}' \prec s'$ such that
$$
\Prb\{ \cW_{(\TT', o)} = \tilde{s}'\} > 0,
$$
then, writing 
$$
\bar{m} := (m, \ldots, m)
$$
for a positive integer $m$ we would have
$$
\Prb\{\cW_{(\TT,v)} = (\tilde{s}' + \overline{1}, s'' + \overline{2k_s - 2})\} > 0
$$
and, by definition of the total order $\prec$,
$$
(\tilde{s}'+ \overline{1}, s'' + \overline{2k_s - 2}) \prec (s' + \overline{1}, s''+ \overline{2k_s - 2}) = s.
$$
This, however, contradicts the minimality of $s$. Similarly, $s''$ is the minimal down-split sequence for $(\TT'', o)$.
By induction, $(\TT', o)$ and $(\TT'', o)$ are uniquely determined.

Since $(\TT,v)$ is obtained by gluing $(\TT',o)$ and $(\TT'', o)$ together at the shared vertex $o$ 
and attaching the marked leaf $v$ to $o$ by an edge,  we see that $(\TT, v)$ is also determined by $s$.
\end{proof}

While the proof of \Cref{prop:det by downsplit} is not in the form of an explicit reconstruction procedure,
the argument  clearly leads to an algorithm for building a marked $3$-valent tree $(\TT, v)$ 
from the corresponding minimal down-split sequence. Namely, $(\TT,v)$ is simply the recursion tree that
results from parsing $s$ as a down-split sequence 
as in \Cref{def:downsplit}, with leaves corresponding to edges that terminate in the sequence $(1)$. 

\begin{figure}[ht]
    \centering
    \includegraphics[width=0.4\textwidth]{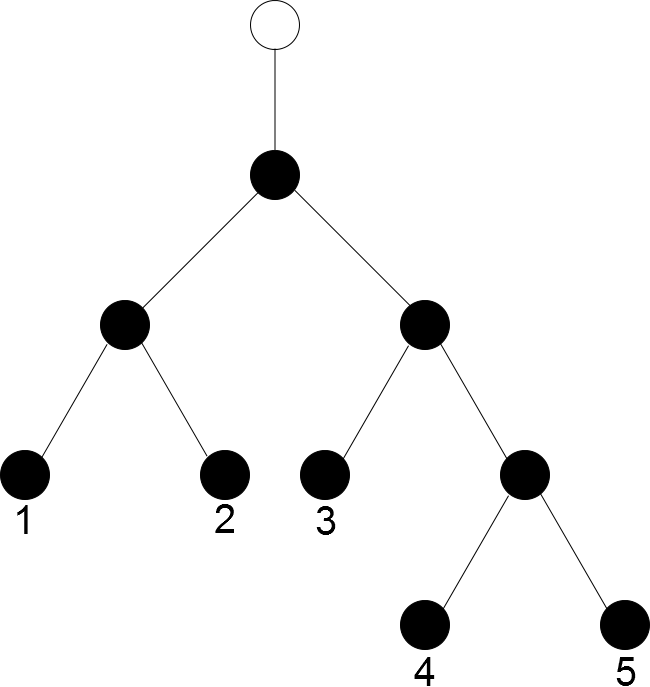}
    \hspace{5 pc}
    \includegraphics[width=0.4\textwidth]{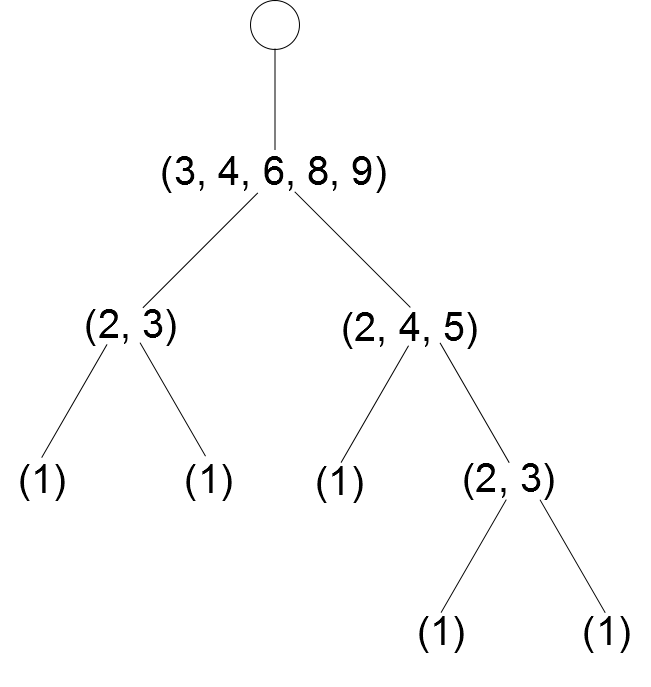}
    \caption{A marked $3$-valent tree with its leaves ordered minimally and the corresponding parse tree for the minimal 
		down-split sequence.}
    \label{fig:ParseTree}
\end{figure}

\subsubsection{Proof of \Cref{thrm:main kary1} and \Cref{thrm:main kary2}}

\label{sec:first half of main proof}

From \Cref{prop:det by downsplit} we are able to easily prove \Cref{thrm:main kary1} for (unmarked) $3$-valent trees.

\begin{proof}
Let $\TT$ be a fixed (unknown) $3$-valent tree with $n$ leaves and let $\cW_{\TT}$ be its random length sequence. 
Conditional on $Y_1$, $\cW_{\TT}$ is the modified random length sequence of the marked binary tree $(\TT, Y_1)$. Thus, if 
$$
\Prb\{\cW_{\TT} = s\} > 0,
$$
then there must be some leaf $v \in \TT$ such that
$$
\Prb\{ \cW_{(\TT, v)} = s \} > 0.
$$
Let $s^*$ be the minimal element of the set
$$
\{ s \text{ down-split } \colon \Prb \{\cW_{\TT} = s\} > 0 \}.
$$
Then $s^*$ must be the minimal down-split sequence for $(\TT, v)$ for at least one leaf $v$ of $\TT$. 
By \Cref{prop:det by downsplit} we can reconstruct $(\TT, v)$ and hence $\TT$ from $s^*$.
\end{proof}

The above argument can be pushed further to prove \Cref{thrm:main kary2} for $\cT$ a random $3$-valent tree.

\begin{proof} 
Let $\cT$ be a random $3$-valent tree with $n$ leaves and 
random length sequence $\cW_{\cT}$.

Given a $3$-valent tree $\TT$ with $n$ leaves, let $s^\TT$ be the minimal element of the set of down-split
sequences of the marked $3$-valent trees $(\TT,v)$ as $v$ ranges over $\LL(\TT)$.
 We equip the set of $3$-valent tree with $n$ leaves with a total order that, with a slight abuse of notation,
we denote by $\prec$ by declaring that $\TT' \prec \TT''$ if $s^{\TT'} \prec s^{\TT''}$. Note
that if $\TT' \prec \TT''$, then $\bP\{\cW_{\TT'} = s^{\TT'}\} > 0$ and $\bP\{\cW_{\TT''} = s^{\TT'}\} = 0$.
Now, for each choice of $\TT$ we have
\[
\bP\{\cW_{\cT} = s^{\TT}\} 
= 
\sum_{\TT'} \bP\{\cT = \TT'\} \bP\{\cW_{\TT'} = s^{\TT}\}
=
\sum_{\TT' \preceq \TT} \bP\{\cT = \TT'\} \bP\{\cW_{\TT'} = s^{\TT}\},
\]
and the conclusion that we can recover $\bP\{\cT = \TT\}$ as $\TT$ ranges over the $3$-valent trees with
$n$ leaves follows simply from the observation that if $b$ is a row vector of length $N$ and
$A$ is an $N \times N$ matrix that has all entries below the diagonal zero and all entries on the
diagonal strictly positive, then there is a unique row vector $x$ of length $N$ such that $b = x A$.
\end{proof}

\subsection{Up-split sequences}

We now prove \Cref{thrm:main kary1} and \Cref{thrm:main kary2} for rooted binary trees. 
Analogous to the objects we introduced for marked $3$-valent trees, we begin with a definition of a class of sequences
that will appear in the support of the random length sequence of a rooted binary tree.

\begin{defn}
An {\em up-split sequence} is an element of the class of increasing sequences of nonnegative
integers defined recursively as follows.

The sequence 
\[
s = (0)
\] 
is an up-split sequence.

a sequence $s = (s_1, \ldots, s_n)$, $n>1$, is an up-split sequence if 
\[
\{1 \leq k < n \colon s_k = 2k - 2 \} \ne \emptyset
\]
and, setting
$$
k_s := \sup \{1 \leq k < n \colon s_k = 2k - 2 \},
$$
\begin{itemize}
\item $(s_1, \ldots, s_{k_s})$ is an up-split sequence,
\item $(s_{k_s + 1} - (2k_s - 1), \ldots, s_n - (2k_s - 1))$ is a down-split sequence.
\end{itemize}
The index $k_s$ is the {\em splitting index} of $s$.
\end{defn}

\begin{ex}
{Suppose that $\TT$ is a rooted binary tree with root $o$.  In a manner similar to the construction in
\Cref{rem:down-split construction} we can order
a partial order on $\VV(\TT)$ by declaring
that $x$ precedes $y$ if $x \ne y$ is on the path between $\rho$ and $y$, and we can extend this
partial order to a total order $<$ such that if $w,x,y,z$ are such that
$w$ and $x$ are not comparable in the partial order  but $w < x$, $w$ precedes $y$
in the partial order,
and $x$ precedes $z$ in the partial order, 
then $y < z$.  
Suppose that $y_1 < y_2 < \ldots < y_n$ is the ordered
listing of $\LL(\TT)$.  Set $s_1:=0$ and
$s_k := \WW_\TT(\{y_1, \ldots, y_k\})$, $2 \le k \le n$.
Then $(s_1, \ldots, s_n)$ is an up-split sequences.
The leaves $y_1, \ldots, y_{k_s}$ and $y_{k_s+1}, \ldots, y_n$ respectively span the two binary subtrees $\TT'$
and $\TT''$ that are rooted at the two
children of the root $o$.  The subtree spanned by $o$ and $y_{k_s+1}, \ldots, y_n$ is a $3$-valent tree.
\begin{figure}[ht]
    \centering
    \includegraphics[width=0.6\textwidth]{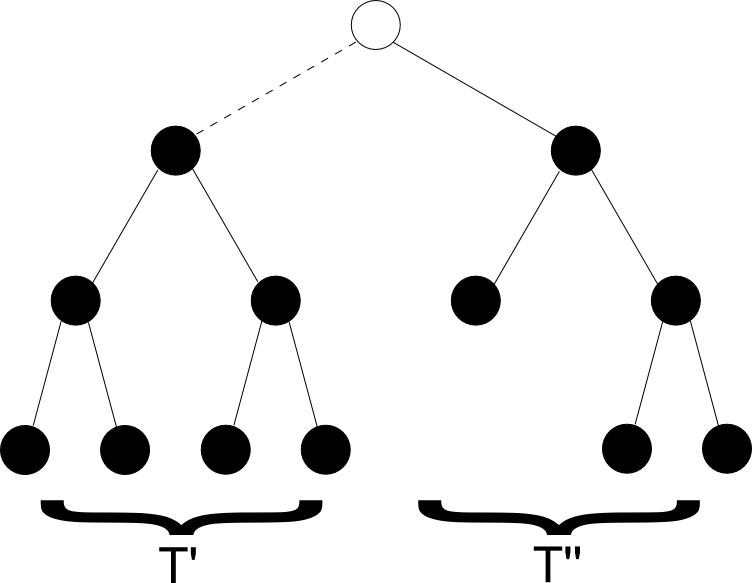}
    \caption{A rooted binary tree split as a rooted binary subtree $\TT'$ and a marked $3$-valent tree $(\TT'', o)$.}
    \label{fig:UpsplitEx}
\end{figure}
}
\label{ex:up-split construction}
\end{ex}

The following analogue of  \Cref{lem:exist of split} is clear from \Cref{ex:up-split construction}.

\begin{lemma}{For every rooted binary tree $\TT$, there is at least one up-split sequence $s$ with
$$
\Prb\{ (0,\cW_{\TT}) = s\} > 0.
$$}
\label{lem:exist upsplit}
\end{lemma}

The following analogue of \Cref{lem:size of downsplit} can be established using a similar inductive proof.

\begin{lemma}{If $s = (s_1, \ldots, s_n)$ is an up-split sequence then $s_n = 2n - 2$.}
\label{lem:size of upsplit}
\end{lemma}

\begin{defn}
Define a total order $\ll$ on the set of up-split sequences of a given length
recursively as follows.  Firstly, $(0) \ll (0)$ does not hold.
Next, let $s$ and $r$ be two up-split sequences indexed by $\{1,\ldots,n\}$
with respective splitting indices 
$k_s$ and $k_r$. Set
$$
s' = \left(s_1, \ldots, s_k \right), 
\quad 
r' = \left( r_1, \ldots, r_k \right)
$$
and
$$
s'' = \left(s_{k + 1} - (2k - 1), \ldots, s_n - (2k - 1) \right), 
\quad
r'' = \left(r_{k + 1} - (2k - 1), \ldots, r_n - (2k - 1) \right).
$$
Declare that
$$
s \ll r
$$
if 
\[
k_s > k_r
\] 
or 
\[
k_s = k_r  \quad \text{and} \quad s' \ll r'
\]
or
\[
k_s = k_r  \quad \text{and} \quad s' = r' \quad
\text{and} \quad s'' \prec r''.
\]
\end{defn}

\begin{rem}
Note for up-split sequences $s$ and $r$, 
that $s \ll r$ implies that the splitting index of $s$ 
is {\em greater than or equal} to the splitting index of $r$. 
For down-split sequences $u$ and $t$,  $u \prec t$ implies that the splitting index of $u$ 
is {\em less than or equal} to the splitting index of $t$.
This change in the direction of the inequalities matches 
the switch in the definition of the splitting index
from an infimum for down-split sequences to a supremum for up-split sequences.
\end{rem}

The next result follows easily by induction.

\begin{lemma}{The binary relation $\ll$ is a total order on the set of 
up-split sequences of a given length.}
\end{lemma}

\begin{defn}
The {\em minimal up-split sequence} for a rooted binary tree $\TT$ is the minimal element 
(with respect to the total order $\ll$) of the set
$$
\{ s \text{ up-split }\colon \Prb \{ \cW_{\TT} = s \} > 0 \}.
$$
\end{defn}

The up-split sequence analogues of \Cref{lem:split subtrees} and \Cref{cor:split subtrees}
are the following and they are proved in essentially the same manner.

\begin{lemma}{Given  a binary tree $\TT$ with root $o$, let $\TT'$ and $\TT''$ be the binary subtrees
rooted at the two children of $o$.
Set
$$
m := \sup \{ 1 \leq k < n \colon \Prb \{ W_k = 2k - 2 \} > 0 \}.
$$
Then $W_m = 2m - 2$ if and only if $Y_1, \ldots, Y_m \in \TT'$ 
and $Y_{m + 1}, \ldots, Y_n \in \TT''$ or {\em vice versa}.}
\label{lem:upsplit subtrees}
\end{lemma}
%
%

\begin{cor}{Let $\TT$ be a rooted binary tree
with random length sequence $\cW_{\TT} = (W_2, \ldots, W_n)$.
Then
\[
m:= \sup\{ 1 \le k  < n \colon \Prb \{ W_k = 2k - 2 \} > 0 \}
\]
is the splitting index for the minimal up-split sequence for $\TT$.}
\label{cor:upsplit subtrees}
\end{cor}

The following analogue of \Cref{prop:det by downsplit} 
for up-split sequences follows from
\Cref{lem:upsplit subtrees} and \Cref{cor:upsplit subtrees} 
in essentially the same manner that \Cref{prop:det by downsplit}
followed from \Cref{lem:split subtrees} and \Cref{cor:split subtrees}.

\begin{prop}{Let $s$ be the minimal up-split sequence for a rooted binary tree $\TT$. 
There is no other rooted binary tree for which $s$ is the minimal up-split sequence.}
\label{prop:det by upsplit}
\end{prop}
%
%
%
%

Clearly, \Cref{prop:det by upsplit} completes the proof of \Cref{thrm:main kary1}. 
To establish \Cref{thrm:main kary2} in the case of $\cT$ 
a random rooted binary tree, we need only repeat the argument of 
the proof of \Cref{thrm:main kary2} given in \Cref{sec:first half of main proof} 
for $3$-valent trees.

\subsection{$(k+1)$-valent and rooted $k$-ary combinatorial trees}
\label{sec:kary}

The proof of the extension \Cref{thrm:main kary1} and \Cref{thrm:main kary2} 
to $(k+1)$-valent and rooted $k$-ary combinatorial trees for $k \geq 3$ 
is very similar to the $k = 2$ case
and involves the introduction of suitable notion of 
down-split and up-split sequences along with appropriate
total orders on these sets of sequences. 
The only difference is that both types of split sequences are now split into $k$ smaller sequences, instead of just two.  We leave the details to the reader.

%
%
%

\section{Open problems}
\label{sec:open probs}

The original conjecture Question~\ref{q:main} remains open in general, both for simple trees with arbitrary edge weights (not in general position), and for combinatorial trees. An even more general question is suggested by \Cref{thrm:main kary2}.

\begin{question}{Let $\cT$ be a random tree with probability distribution supported 
either on the set of simple trees with $n$ leaves and general edge weights or the set of combinatorial trees with $n$ leaves. 
Can the probability distribution of $\cT$ be determined uniquely from the joint probability distribution of 
the random length sequence $\cW_{\cT}$?}
\label{q:open general}
\end{question}

Even if the answer to Question~\ref{q:open general} is ``no'', the answer may
still be ``yes'' if the probability distribution of 
$\cT$ is known {\em a priori} to belong to
some particular family of probability distributions. There are,
of course, many families of
probability models for with random trees with $n$ leaves that
are described by a small number of parameters (for example,
conditioned Galton-Watson models, the various preferential attachment
models), and perhaps the value of these parameters can be determined
from the joint probability distribution of the random length sequence
of a random tree that is known {\em a priori} to be distributed according to a
member of one of these families.



\begin{question}{What
are the necessary and sufficient conditions on a vector
for there to be an edge-weighted tree $\TT$ such that
the vector is in the support of $\cW_\TT$?}
\end{question}

We remarked in the Introduction that the focus of this paper is superficially
similar to that in \cite{MR786484}, where the problem of reconstructing
a combinatorial tree from its number deck (the sizes of the subtrees in the forests
produced by deleting each vertex) was studied.  The lists of lists that
are the number deck of some combinatorial tree are characterized in \cite{MR846676}.

\begin{question}{Are there more parsimonious quantities derived from the
joint probability distribution of the random length sequence
that still carry a lot of information about $\TT$?
For example, how much information about $\TT$ is contained in the expectation 
$(\bE[W_2], \ldots, \bE[W_n])$ of the random length sequence
and is it possible to characterize those vectors which can arise as the
expectation of the random length sequence?}
\end{question}

\providecommand{\bysame}{\leavevmode\hbox to3em{\hrulefill}\thinspace}
\providecommand{\MR}{\relax\ifhmode\unskip\space\fi MR }
\providecommand{\MRhref}[2]{%
  \href{http://www.ams.org/mathscinet-getitem?mr=#1}{#2}
}
\providecommand{\href}[2]{#2}

\end{document}